\documentclass[11pt]{amsart}

\usepackage{enumerate, amsmath, amsthm, amsfonts, amssymb, xy,  mathrsfs, graphicx, paralist, lscape}
\usepackage[usenames, dvipsnames]{color}
\usepackage[margin=1in]{geometry} 
\usepackage{tikz}
\usetikzlibrary{matrix,arrows}

\numberwithin{equation}{section}
\newtheorem{theorem}[equation]{Theorem}
\newtheorem*{thm}{Theorem}
\newtheorem{proposition}[equation]{Proposition}
\newtheorem{lemma}[equation]{Lemma}
\newtheorem{corollary}[equation]{Corollary}
\newtheorem*{cor}{Corollaries}

\theoremstyle{definition}
\newtheorem{rmk}[equation]{Remark}
\newenvironment{remark}[1][]{\begin{rmk}[#1] \pushQED{\qed}}{\popQED \end{rmk}}
\newtheorem{rmks}[equation]{Remarks}

\newtheorem{eg}[equation]{Example}
\newenvironment{example}[1][]{\begin{eg}[#1] \pushQED{\qed}}{\popQED \end{eg}}
\newtheorem{defn}[equation]{Definition}
\newenvironment{definition}[1][]{\begin{defn}[#1]\pushQED{\qed}}{\popQED \end{defn}}
\newtheorem{ques}[equation]{Question}

\newtheorem{notn}[equation]{Notation}

\usepackage[in]{fullpage}
\usepackage[colorlinks=true, pdfstartview=FitV, linkcolor=blue, citecolor=blue, urlcolor=blue]{hyperref}

\newcommand{\bb}{\mathbf{b}}

\newcommand{\bd}{\mathbf{d}}

\newcommand{\br}{\mathbf{r}}

\newcommand{\bw}{\mathbf{w}}

\newcommand{\zO}{\ensuremath{\Omega}}
\newcommand{\za}{\ensuremath{\alpha}}
\newcommand{\zb}{\ensuremath{\beta}}

\newcommand{\zd}{\ensuremath{\delta}}
\newcommand{\zg}{\ensuremath{\gamma}}

\newcommand{\II}{\mathbb{I}}

\renewcommand{\phi}{\varphi}

\renewcommand{\tilde}[1]{\widetilde{#1}}

\newcommand{\comment}[1]{}

\makeatletter
\def\Ddots{\mathinner{\mkern1mu\raise\p@
\vbox{\kern7\p@\hbox{.}}\mkern2mu
\raise4\p@\hbox{.}\mkern2mu\raise7\p@\hbox{.}\mkern1mu}}
\makeatother

\DeclareMathOperator{\coker}{coker}
\newcommand{\Hom}{\operatorname{Hom}}

\DeclareMathOperator{\rank}{rank}

\DeclareMathOperator{\Spec}{Spec}

\DeclareMathOperator{\rep}{rep}
\DeclareMathOperator{\Mat}{Mat}

\newcommand{\GL}{\mathbf{GL}}

\newcommand{\cell}{X_\circ^\bw}
\newcommand{\mcell}{Y_\circ^{\bw}}
\newcommand{\tq}{\tilde{Q}}
\newcommand{\tv}{\tilde{V}}
\newcommand{\td}{\tilde{\bd}}
\newcommand{\bid}{\mathbf{1}}

\begin{document}

\title{Type $A$ quiver loci and Schubert varieties}

\author{Ryan Kinser}
\author{Jenna Rajchgot}
\address{Northeastern University, Department of Mathematics, Boston, MA, USA}
\email[Ryan Kinser]{r.kinser@neu.edu}
\address{University of Michigan, Department of Mathematics, Ann Arbor, MI, USA}
\email[Jenna Rajchgot]{rajchgot@umich.edu}
\thanks{The first author was supported by NSA Young Investigator Grant H98230-12-1-0244.}
\dedicatory{Dedicated to the memory of Andrei Zelevinsky}

\begin{abstract}
We describe a closed immersion from each representation space of a type $A$ quiver with bipartite (i.e., alternating) orientation to a certain opposite Schubert cell of a partial flag variety. This ``bipartite Zelevinsky map'' restricts to an isomorphism from each orbit closure to a Schubert variety intersected with the above-mentioned opposite Schubert cell. For type $A$ quivers of arbitrary orientation, we give the same result up to some factors of general linear groups.

These identifications allow us to recover results of Bobi\'nski and Zwara; namely we see that orbit closures of type $A$ quivers are normal, Cohen-Macaulay, and have rational singularities.  We also see that each representation space of a type $A$ quiver admits a Frobenius splitting for which all of its orbit closures are compatibly Frobenius split.
\end{abstract}

\maketitle

\setcounter{tocdepth}{1}
\tableofcontents

\section{Introduction}

\subsection{Context and History}\label{sect:history}
A quiver is a finite directed graph, and a representation of a quiver is a choice of vector space for each vertex and linear map for each arrow.  When the underlying graph is a type $A$ Dynkin diagram, we say the quiver is of type $A$. 
Once the vector spaces at each vertex are fixed, the collection of representations is an algebraic variety.  This ``representation space'' carries the action of a base change group.  A brief review of quiver representations is in Sections \ref{sect:quivers} through \ref{sect:repspaces}.

A type $A$ quiver with all arrows in the same direction is called equioriented, and the study of orbit closures (a.k.a. ``quiver loci'') for these quivers has a long and rich history.  Here, a representation space consists of all sequences of matrices $(M_1, \cdots, M_n)$ where $M_i$ determines a linear map from $K^{d_{i-1}}$ to $K^{d_i}$: 
\begin{equation}
K^{d_0} \xrightarrow{M_1} K^{d_1} \xrightarrow{M_2} \cdots  \xrightarrow{M_n} K^{d_n}.
\end{equation}
In this case, a quiver locus is described by imposing conditions on the ranks of all possible products of one or more of these matrices.

Specifying that each product of two consecutive matrices is zero determines a union of quiver loci known as a ``Buchsbaum-Eisenbud variety of complexes''. Varieties of complexes were studied extensively starting in the 1970s, and, when irreducible, they were shown to be normal, Cohen-Macaulay, and have rational singularities \cite{MR0396528, MR0384817, DCSvarietyofcomplexes, MR1684146, MR738917}.  Abeasis, del Fra, and Kraft extended these results to all equioriented type $A$ quiver loci in characteristic~0~\cite{MR626958}.


Soon after, a connection between equioriented type $A$ quiver loci and Schubert varieties in type $A$ flag varieties started to be uncovered: Musili and Seshadri noticed that Buchsbaum-Eisenbud varieties of complexes could be realized as open sets of unions of Schubert varieties. This allowed them to transport techniques such as standard monomial theory to study varieties of complexes \cite{MR717615}. For arbitrary rank conditions of equioriented type $A$ quivers, Zelevinsky gave an explicit set-theoretic identification of quiver loci with open subsets of Schubert varieties \cite{Zgradednilp}; this is now called the (equioriented) ``Zelevinsky map''. Lakshmibai and Magyar later showed that the Zelevinsky map is a scheme-theoretic isomorphism \cite{LMdegen}. Knutson, Miller, and Shimozono made great use of Zelevinsky's map to paint several beautiful combinatorial pictures of the torus equivariant cohomology of type $A$ quiver loci \cite{KMS}. Their paper is the main inspiration for our work.

For type $A$ quivers of arbitrary orientation,
Bobi\'nski and Zwara showed that orbit closures are normal and Cohen-Macaulay, with rational singularities \cite{BZ-typeA}, a result that is especially important for producing $K$-theoretic formulas for quiver loci \cite{MR1932326, MR2137947,MR2492443}. 
Bobi\'nski and Zwara's technique is to use Auslander-Reiten theory to construct ``Hom-controlled functors''  \cite{MR1888422}. These functors ensure that any singularity type appearing for an arbitrary orientation appears for a (typically larger) equioriented quiver, and thus also for a Schubert variety (by the Zelevinsky map). In fact, they later showed that the singularity types appearing in type $A$ quiver loci exactly coincide with singularity types of Schubert varieties in type $A$ flag varieties \cite{MR1967381}.  

We end by remarking that quiver loci for Dynkin quivers are important in the study of degeneracy locus formulas, a line of investigation initiated by Buch and Fulton \cite{BFchernclass} to generalize the classical Thom-Porteous formula.  
They are also important in Lie theory, where they lie at the foundation of Lusztig's geometric realization of Ringel's work on quantum groups \cite{MR1035415,Rhallalgebras}.

\subsection{Main results}\label{sect:mainresults}
In this paper we treat the bipartite orientation (i.e., every vertex is either a source or sink) as fundamental.  This is in contrast with previous approaches, which reduce problems for arbitrary orientations to the equioriented setting.

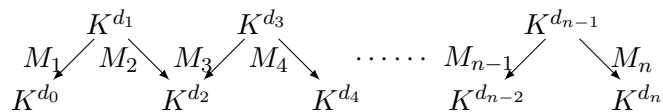
\[
\hbox{\begin{tikzpicture}[point/.style={shape=circle,fill=black,scale=.5pt,outer sep=3pt},>=latex]
   \node[outer sep=-2pt] (y0) at (-1,0) {$K^{d_0}$};
   \node[outer sep=-2pt] (x1) at (0,1) {$K^{d_1}$};
  \node[outer sep=-2pt] (y1) at (1,0) {$K^{d_2}$};
   \node[outer sep=-2pt] (x2) at (2,1) {$K^{d_3}$};
  \node[outer sep=-2pt] (y2) at (3,0) {$K^{d_4}$};
  \node (dots) at (3.75, 0.5) {$\cdots\cdots$};
  \node[outer sep=-2pt] (y3) at (5,0) {$K^{d_{n-2}}$};
  \node[outer sep=-2pt] (x4) at (6,1) {$K^{d_{n-1}}$};
  \node[outer sep=-2pt] (y5) at (7,0) {$K^{d_n}$};
  \path[->]
  	(x1) edge node[left] {$M_1$} (y0) 
	(x1) edge node[left] {$M_2$} (y1)
  	(x2) edge node[left] {$M_3$} (y1) 
	(x2) edge node[left] {$M_4$} (y2)
  	(x4) edge node[left] {$M_{n-1}$} (y3) 
  	(x4) edge node[right] {$M_n$} (y5); 
	
\node at (3,-.75) {A representation of a bipartite, type $A$ quiver, with dimension vector $\bd = (d_0, \dotsc, d_n)$};
   \end{tikzpicture}}
   \]

Our first main result is the construction of a Zelevinsky map in this setting.
The precise formulation is in Section \ref{sect:qrcToNW}; here we summarize its properties (see Theorem \ref{thm:mainThmSection3} for details).
\begin{thm}
Let $\rep_Q(\bd)$ be a space of representations of a bipartite quiver of type $A$, having fixed dimension vector $\bd$.  Then there exists an opposite Schubert cell $Y$ in a partial flag variety, and a closed immersion
\begin{equation}
\zeta \colon \rep_Q(\bd) \to Y
\end{equation}
which identifies each orbit closure in $\rep_Q(\bd)$ with a Schubert variety intersected with $Y$.
This identification can be realized combinatorially by the association of a ``bipartite Zelevinsky permutation'' to each orbit.
\end{thm}

Relating the geometry of quiver loci of arbitrarily oriented type $A$ quivers to those in the bipartite case is quite natural; this is our second main result.  The simplicity of our approach is in stark contrast to trying to reduce the geometry of arbitrary orientations to the equioriented case, which seems to be quite difficult.  We paraphrase our result here, with the detailed statement found in Theorem \ref{thm:arbitraryorient}.

\begin{thm}
Let $\rep_Q(\bd)$ be a representation space of an arbitrarily oriented quiver of type $A$.  Then there exists an open subset $U$ of a representation space of a certain bipartite quiver, along with a
smooth, $\GL(\bd)$-equivariant projection $\pi\colon U \to \rep_Q(\bd)$ that induces a containment preserving bijection on quiver loci.
\end{thm}
 
From these two theorems, we recover one of Bobi\'nski and Zwara's results, and find a nice Frobenius splitting of each representation space of a type $A$ quiver.
\begin{cor}  Let $Q$ be a type $A$ quiver of arbitrary orientation, and let $\bd$ be a dimension vector for $Q$.  Then the following hold.
\begin{enumerate}
\item \cite[Thm. 1.1]{BZ-typeA} All orbit closures in $\rep_Q(\bd)$ are normal and Cohen-Macaulay. When working over a field of characteristic 0, they also have rational singularities.
\item Over a perfect field of positive characteristic, there exists a Frobenius splitting of $\rep_Q(\bd)$ that simultaneously compatibly splits all orbit closures.
\end{enumerate}
\end{cor}

As mentioned above, Bobi\'nski and Zwara obtain their results via powerful representation-theoretic technology.  Their approach allows the expert in finite-dimensional algebras to quickly see that singularity types are independent of orientation for type $A$ and $D$ quivers, using only representation theory.
In contrast, our methods are specific to type $A$ quivers, but more explicit in realizing the entire quiver locus as a variety which is already well understood.  In particular, our method is naturally suited to the use of combinatorics in describing non-local properties of quiver loci such as their torus-equivariant cohomology classes and orbit closure containment (i.e., degeneration order).  

\begin{remark}\label{rmk:access}
This paper is meant to be self-contained and accessible to both combinatorial algebraic geometers and representation theorists.  For this reason, some background and proofs have been included that may seem overly detailed or trivial to experts in one field or the other.
\end{remark}

\subsection*{Acknowledgements}
We especially thank Allen Knutson for showing us that our original statement of our first main result in terms of positroid varieties could instead be formulated in terms of Schubert varieties, and we also thank him for looking at a first version of this paper. We thank Karen E. Smith for explaining how to extend our Frobenius splitting statements in the bipartite type $A$ setting to the general type $A$ setting, and we thank Shrawan Kumar and V. Lakshmibai for enlightening conversations. This work originated when the authors met in August 2012 at the joint introductory workshop on cluster algebras and commutative algebra at MSRI in Berkeley, CA.  We also thank Daniel Hern\'andez for participation in early discussions there.
\section{Background}\label{sect:background}
Let $K$ be a field which will remain fixed throughout the paper, thus usually omitted from the notation.  

\subsection{Quiver representations}\label{sect:quivers}
A \textbf{quiver} $Q$ is a finite directed graph. The set of vertices of $Q$ is denoted by $Q_0$, and the set of arrows is denoted by $Q_1$.  The vertex at the tail (starting point) of an arrow $a\in Q_1$ is denoted by $ta$, and the vertex at the head (ending point) is denoted by $ha$. 

A \textbf{representation} $V$ of a quiver $Q$ is an assignment of a finite-dimensional vector space $V_z$ to each vertex $z\in Q_0$, and a linear map $V_a:V_{ta}\rightarrow V_{ha}$ to each arrow $a\in Q_1$. There is a natural notion of morphism between two representations of the same quiver.  The collection of all representations of a fixed quiver $Q$ is equivalent to the category of finite-dimensional modules over the ``path algebra'' of $Q$, so all the standard operations on modules make sense for quiver representations.  See the text \cite{assemetal} for basics of quiver representations.  


Let $Q$ be a quiver of type $A$ with bipartite orientation, so that every vertex is either a source or a sink.  We can always assume that $Q$ has an odd number of vertices.\footnote{To cover quivers with an even number vertices, say $2n$, we simply work with representations of the $Q$ in \eqref{eq:typea} which have dimension zero at vertex $y_n$. This convention allows us to simplify the presentation.  Maps to or from zero-dimensional vector spaces are represented by matrices with 0 rows or columns, resp.}  So, by dualizing if necessary, we can assume that each endpoint has an incoming arrow and we label the vertices and arrows of $Q$ as follows:
\begin{equation}\label{eq:typea}
\vcenter{\hbox{\begin{tikzpicture}[point/.style={shape=circle,fill=black,scale=.5pt,outer sep=3pt},>=latex]
   \node[outer sep=-2pt] (y0) at (-1,0) {$y_0$};
   \node[outer sep=-2pt] (x1) at (0,1) {$x_1$};
  \node[outer sep=-2pt] (y1) at (1,0) {$y_1$};
   \node[outer sep=-2pt] (x2) at (2,1) {$x_2$};
  \node[outer sep=-2pt] (y2) at (3,0) {$y_2$};
  \node (dots) at (3.75, 0.5) {$\cdots\cdots$};
   \node[outer sep=-2pt] (x3) at (5,1) {$x_{n-1}$};
  \node[outer sep=-2pt] (y3) at (6,0) {$y_{n-1}$};
  \node[outer sep=-2pt] (x4) at (7,1) {$x_n$};
  \node[outer sep=-2pt] (y5) at (8,0) {$y_n$};
  \path[->]
  	(x1) edge node[left] {$\za_1$} (y0) 
	(x1) edge node[left] {$\zb_1$} (y1)
  	(x2) edge node[left] {$\za_2$} (y1) 
	(x2) edge node[left] {$\zb_2$} (y2)
	(x3) edge node[left] {$\zb_{n-1}$} (y3)
  	(x4) edge node[left] {$\za_n$} (y3) 
  	(x4) edge node[left] {$\zb_n$} (y5); 
   \end{tikzpicture}}} .
\end{equation}

We write $J \subseteq Q$ to denote that $J$ is an \textbf{interval} in $Q$, that is, a connected subquiver of $Q$.  Then the indecomposable representations of $Q$ are in bijection with intervals in $Q$.  Explicitly, let $\II_J$ be the representation defined at each vertex $z$ by
\begin{equation}
(\II_J)_z = 
\begin{cases}
K & z \in J\\
0 & \text{otherwise}\\
\end{cases},
\end{equation}
with the identity map for each arrow in $J$, and the other maps zero.  By performing Gaussian elimination alternately on rows and columns down the quiver, one can see that these constitute a complete set of isomorphism classes of indecomposable representations of $Q$ \cite{gabriel}.  The Krull-Schmidt property of $\rep(Q)$ implies that each $V \in \rep(Q)$ has an essentially unique expression as 
\begin{equation}\label{eqn:KS}
V \simeq \bigoplus_{J \subseteq Q} s_J\, \II_J, \quad s_J \in \mathbb{Z}_{\geq 0}
\end{equation}
where $s_{J}$ is the multiplicity of the summand $\II_J$ in $V$.

\subsection{Projective representations}\label{sect:proj}
To study quiver representations from a geometric point of view, we want to encode the data of the Krull-Schmidt decomposition \eqref{eqn:KS} into rank conditions on certain matrices.
To do this, we will need to replace certain representations with projective presentations, thus encoding the isomorphism class of a representation into a single matrix.
We review here the basic facts on projective representations of quivers (see \cite[\S III.2]{assemetal} for more detail).

The indecomposable projective representations of any quiver are in bijection with its vertices, and we write $P(z)$ for the projective associated to a vertex $z$.  Concretely, the vector space which $P(z)$ associates to a vertex $v$ has the set of paths from $z$ to $v$ as its basis.  The map over an arrow $a$ takes a path and concatenates the arrow $a$ to it.
Given two vertices $v, w$, the set of paths from $w$ to $v$ can be naturally identified with a basis of the vector space $\Hom_Q(P(v), P(w))$.  Therefore, a map between arbitrary projectives is given by a matrix whose entries are linear combinations of paths in $Q$.

The projective $P(z)$ represents the functor ``restrict to the vertex $z$''.  That is, there is a functorial isomorphism $\Hom_Q(P(z), V) \cong V_z$.  Furthermore, when $\Hom_Q(?, V)$ is applied to the morphism $P(z) \xrightarrow{a} P(w)$ associated to an arrow $w \xrightarrow{a} z$ in $Q$, we get the following commutative diagram of vector spaces.
\begin{equation}
\vcenter{\hbox{\begin{tikzpicture}
\node (ul) at (0,2) {$\Hom_Q (P(w), V)$};
\node (ur) at (5,2) {$\Hom_Q (P(z), V)$};
\node (ll) at (0,0) {$V_w$};
\node (lr) at (5,0) {$V_z$};
\path[->]
(ul) edge node[above] {$\Hom_Q (a, V)$} (ur)
(ll) edge node[above] {$V_a$} (lr)
(ul) edge node[left] {$\cong$} (ll)
(ur) edge node[right] {$\cong$} (lr);
\end{tikzpicture}}}
\end{equation}
An arbitrary map between projectives $\Phi\colon P^1 \to P^0$ is given by a matrix whose entries are linear combinations of paths. Applying $\Hom_Q (?, V)$ replaces each entry of the matrix of $\Phi$ with the corresponding linear combination of maps appearing over arrows of $Q$ in $V$.

\subsection{Representation spaces of a quiver}\label{sect:repspaces}
A \textbf{dimension vector} $\bd \colon Q_0 \to \mathbb{Z}_{\geq 0}$ for a quiver $Q$ is an assignment of a nonnegative integer to each vertex of $Q$.  For a fixed $\bd$, define the associated \textbf{representation space} to be
\begin{equation}
\rep_Q(\bd) := \prod_{a \in Q_1} \Mat_{\bd(ha), \bd(ta)}(K),
\end{equation}
where $\Mat_{m, n}(K)$ denotes the algebraic variety of matrices with $m$ rows, $n$ columns, and entries in $K$. Each $V=(V_a)_{a \in Q_1}$ in $\rep_Q(\bd)$ gives a representation of $Q$, and so $\rep_Q(\bd)$ parametrizes representations of $Q$ with vector space $K^{\bd(z)}$ at vertex $z\in Q_0$. 

%

There is a \textbf{base change group}
\begin{equation}
\GL(\bd) := \prod_{z \in Q_0} \GL_{\bd(z)}(K)
\end{equation}
whose action at each vertex induces an action on $\rep_Q(\bd)$. Explicitly, if $g = (g_z)_{z\in Q_0}$ is an element of $\GL(\bd)$, and $V = (V_a)_{a\in Q_1}$ is an element of $\rep_Q(\bd)$, then 
\begin{equation}
g \cdot V = (g_{ha}V_ag_{ta}^{-1})_{a\in Q_1}. 
\end{equation}
Two points $V, W \in \rep_Q(\bd)$ lie in the same orbit if and only if $V$ and $W$ are isomorphic as representations of $Q$.

Clearly, $\rep_Q(\bd)$ is isomorphic to affine space of dimension \[\sum_{a\in Q_1}\bd(ha)\times \bd(ta),\] and the coordinate ring $K[\rep_Q(\bd)]$ is generated by the coordinates that pick out the matrix entries. In the case of interest in this paper, where $Q$ is bipartite of type $A$, labeled as in \eqref{eq:typea}, we write $A_i$ and $B_i$ for matrices over $\za_i$ and $\zb_i$ with variable entries. 
That is, these matrices have entries in $K[\rep_Q(\bd)]$ such that evaluating $A_i$ (resp. $B_i$) at a point $V \in \rep_Q(\bd)$ gives the matrix $V_{\za_i}$ (resp. $V_{\zb_i}$) over the arrow $\za_i$ (resp. $\zb_i$). 


\subsection{Schubert varieties}\label{sect:schubert}
Here we present the basic facts about Schubert varieties that we need, following our main reference \cite[\S1]{KMS}. Throughout, $Q$ denotes a bipartite quiver of type $A$, labeled as in \eqref{eq:typea}. We fix a dimension vector $\bd$, and let $d_x := \sum_i \bd(x_i)$, $d_y := \sum_i \bd(y_i)$, and $d := d_x + d_y$. 

Let $G:=\GL_d(K)$ and let $P$ be the parabolic subgroup of block \emph{lower} triangular matrices where the diagonals have block sizes $\bd(y_0), \bd(y_1), \dotsc, \bd(y_n), \bd(x_n), \bd(x_{n-1}), \dotsc, \bd(x_1)$.  Let $B^+$ (resp. $B^-$) denote the subgroup of upper (resp. lower) triangular matrices in $G$. \textbf{Schubert cells} are $B^+$-orbits in the partial flag variety $P\backslash G$ for the $B^+$-action by right multiplication, and \textbf{Schubert varieties} are the closures of these orbits. Analogously, \textbf{opposite Schubert cells} are $B^-$-orbits and \textbf{opposite Schubert varieties} are their closures. 

Let $W:=S_d$, the symmetric group on $d$ letters, and consider a permutation $v \in W$ as a matrix with a 1 in position $(i, v(i))$, and zeros elsewhere.   Let $W_P := S_{\bd(y_0)}\times \cdots \times S_{\bd(x_1)}$ (a standard parabolic subgroup), considered as a subgroup of $W$ consisting of permutation matrices down the block diagonal.
For any $v\in W$, the coset $(W_P)v$ has a unique element of minimal length\footnote{Recall that the length $l(v)$ of $v\in S_d$ is the number of pairs $(i,j)\in \{1,\dots,d\}\times \{1,\dots,d\}$ with the property that $i<j$ but $v(i)>v(j)$.} \cite[Cor.~2.4.5]{MR2133266}. Let $W^P$ denote the set of these minimal length coset representatives. Schubert (or opposite Schubert) cells are indexed by the elements in $W^P$; given $v\in W^P$, let $X_v^\circ$ denote the Schubert cell $P\backslash PvB^+$, and let $X_v$ denote its closure. Similarly, let $X^v_\circ$ denote the opposite Schubert cell $P\backslash PvB^{-}$, and let $X^v$ denote its closure.  

Throughout this paper, we will be interested in the permutation
\begin{equation}\label{eq:permCell}
\bw:= \begin{pmatrix}
0 & \bid_{d_y}\\
\bid_{d_x} & 0 \\
\end{pmatrix}\in W^P,
\end{equation}
where $\bid_{d_y}$ denotes a size $d_y$ identity matrix.
The opposite Schubert cell $\cell = P\backslash P\bw B^{-}$ is isomorphic to the space of matrices of the form 
\begin{equation}\label{eq:cell}
\begin{pmatrix}
* & \bid_{d_y}\\
\bid_{d_x} & 0 \\
\end{pmatrix}
\end{equation}
\comment{
\begin{equation}\label{eq:cell}
Z= \begin{pmatrix}
\begin{tikzpicture}[every node/.style={minimum width=1.5em}]
\matrix (m0) [matrix of math nodes]
{
& & &\bid_{\bd(y_0)} &  & \\
& & & & \ddots & \\
& & & &  &\bid_{\bd(y_n)} \\
\bid_{\bd(x_n)} &  & \\
 & \ddots & \\
&  &\bid_{\bd(x_1)} \\
};
\node[scale=4] (star) at (m0-5-2 |- m0-2-5) {*};
\node[scale=3] (zero) at (m0-5-2 -| m0-2-5)  {$0$};
\draw (m0-4-1.north west) -- (m0-3-6.south east);
\draw (m0-1-4.north west) -- (m0-6-3.south east);
\end{tikzpicture}
\end{pmatrix}
\end{equation}} 
where $*$ denotes a block of arbitrary entries. We name the space of these matrices $\mcell$ and note that the isomorphism from $\mcell$ to $\cell$ is the map which sends a matrix to its coset mod $P$. 

Let $Z$ be a matrix of the form shown in (\ref{eq:cell}) that has indeterminates in the block labelled~$*$. 
Let $v\in W^P$.  From \cite[\S6]{Fulton} (see also \cite{MR2422304}, and \cite[\S1.3]{KMS}), the intersection $X_v \cap \cell$ is isomorphic to a subvariety of $\mcell$ obtained by imposing conditions on the ranks of certain submatrices of $Z$. Following \cite{MR2422304}, we call an intersection of a Schubert variety with an opposite Schubert cell a \textbf{Kazhdan-Lusztig variety}.  Let $Z_{p\times q}$ denote the northwest submatrix of $Z$ consisting of the top $p$ rows and left $q$ columns. 
Let \[I_v:=\langle \text{minors of size } (1+\rank v_{p\times q}) \textrm{ in } Z_{p\times q}~|~(p,q)\in\{1,\dots,d\}\times\{1,\dots,d\}\rangle.\] Then, $X_v\cap\cell$ is isomorphic to the subvariety \[Y_v:=\Spec K[Y^{\bw}_\circ]/I_v.\] Note that $Y_v$ is non-empty if and only if $v \leq \bw$ in Bruhat order. When $Y_v$ is non-empty, it has dimension $l(\bw)-l(v) = d_xd_y-l(v)$, where $l(v)$ denotes the length of the permutation $v\in W^P$.  

We end this section by recalling another useful result of Fulton; the ideal $I_v$ has a much smaller generating set than the one given above. 
To a permutation $v\in S_d$, assign a $d\times d$ grid with a $\times$ placed in position $(i,v(i))$. The set of locations (or boxes) in the grid that have a $\times$ neither directly north nor directly west is the \textbf{diagram} of $v$. The number of boxes in the diagram is the length of $v$. Fulton's \textbf{essential set} $\mathcal{E}ss(v)$ is the set of those $(i,j)\in \{1,\dots,d\}\times \{1,\dots,d\}$ such that neither $(i+1,j)$ nor $(i,j+1)$ is in the diagram of $v$. By \cite[\S3]{Fulton}, we have
\[I_v = \langle \text{minors of size}\ (1+\rank v_{p\times q})\textrm{ in } Z_{p\times q}~|~ (p,q)\in \mathcal{E}ss(v)\rangle.\] 


\section{Describing orbits using rank conditions}\label{sect:qrc}

\subsection{From orbits to ranks of matrices}
In this section, we construct a collection of matrices with the property that the ranks of these matrices completely determine orbits in $\rep_Q(\bd)$. These matrices naturally arise from minimal projective resolutions of the (non-projective) indecomposable representations of $Q$. 

Define $M_Q$ to be the matrix with entries in $K[\rep_Q(\bd)]$ built in block form as 
\begin{equation}\label{eq:M}
M_Q=
\begin{pmatrix}
& & & & A_1\\
& & & A_{2} & B_1\\
& & A_{3} & B_{2}\\
& \Ddots&  \Ddots & \\
A_n & B_{n-1} \\
B_n\\
\end{pmatrix},
\end{equation}
with 0 entries in the unlabeled blocks. (Recall that in the last paragraph of Section \ref{sect:repspaces} we have associated to each arrow $\zg$ of $Q$ a matrix with entries in $K[\rep_Q(\bd)]$.)
For an interval $J=[\zg, \zd] \subseteq Q$ having leftmost arrow $\zg$ and rightmost arrow $\zd$,\footnote{For an interval $\{v\}$ with no arrows, one should take $M_{\{v\}}$ to have $\bd(v)$ rows and 0 columns.}
 let $M_J$ be the submatrix of $M_Q$ whose upper right block is associated to $\zg$ and whose lower left block is associated to $\zd$.  For example, we have

\begin{equation}\label{eq:Megs}
M_{[\za_3, \zb_6]}=
\begin{pmatrix}
& & & A_3\\
& & A_{4} & B_3 \\
& A_{5} & B_{4}\\
A_6 & B_{5} \\
B_6\\
\end{pmatrix} \quad \text{and} \quad
M_{[\zb_2, \za_4]}=
\begin{pmatrix}
 & A_{3} & B_2 \\
 A_{4} & B_{3}\\
\end{pmatrix}.
\end{equation}
Each $M_{J}$ has entries in the coordinate ring $K[\rep_Q(\bd)]$, and evaluating at $V \in \rep_Q(\bd)$ has the effect of replacing each $A_k$ with $V_{\za_k}$ and $B_k$ with $V_{\zb_k}$.

The rank of such a matrix is invariant under action of the base change group, so each interval $J$ defines a function which is constant on orbits.
\begin{equation}\label{eq:rankfunctions}
r_{J} \colon \rep_Q(\bd) \to \mathbb{Z}_{\geq 0}, \quad r_{J}(V) = \rank M_{J}(V)
\end{equation}



\begin{proposition}\label{prop:rankorbit}
Two representations $V, W \in \rep_Q(\bd)$ lie in the same $\GL(\bd)$-orbit if and only if $r_J(V) = r_J(W)$ for all intervals $J \subseteq Q$.
\end{proposition}
\begin{proof}
We prove the proposition by showing that there is a bijection $f$ of the set of intervals in $Q$ with itself such that, for any $V \in \rep_Q(\bd)$ and interval $J \subseteq Q$, we have
\begin{equation}\label{eq:rankhom}
r_J(V) + \dim \Hom_Q(\II_{f(J)}, V) = \text{ a constant depending on }\bd \text{ but not }V .
\end{equation}
Given this, the following result of Auslander then implies that knowing these ranks is equivalent to knowing the orbit of a representation:  $V$ and $W$ lie in the same orbit if and only if $\dim \Hom_Q(X, V) = \dim \Hom_Q(X, W)$ for all indecomposable $X \in \rep(Q)$ \cite{Auslander:1982fk}.

Given an interval $J$, we replace all matrices $A_i, B_i$ in $M_J$ with the formal arrows $\za_i, \zb_i$ to get a matrix $\Phi_J$. Then, as reviewed in Section \ref{sect:proj}, this defines morphism $\Phi_J\colon P^1 \to P^0$ between two projective representations of $Q$.\footnote{An interval containing a single vertex $v$ and no arrows corresponds to the map $0 \to P(v)$ with projective cokernel.}  The explicit decomposition of $P^1$ and $P^0$ into indecomposables is easy but notationally cumbersome, as it depends on whether each of the endpoints of $J$ is of type $x$ or $y$.

For example, the case that $J$ is of the form $J=[\za_i, \zb_j]$ gives
\begin{equation}
\Phi_{[\za_i, \zb_j]}=
\begin{pmatrix}
& & & & \za_i\\
& & & \za_{i+1} & \zb_i\\
& & \za_{i+2} & \zb_{i+1}\\
& \Ddots&  \Ddots & \\
\za_j & \zb_{j-1} \\
\zb_j\\
\end{pmatrix}.
\end{equation}
which is a homomorphism
\begin{equation}
 P^1 = \bigoplus_{k=i-1}^j P(y_k)  \xrightarrow{\Phi_{[\za_i, \zb_j]}} P^0 = \bigoplus_{k=i}^j P(x_k) .
\end{equation}
Then it is straightforward to see that $\coker \Phi_{[\za_i, \zb_j]} \simeq \II_{[\zb_i, \za_j]}$.  In essence, each row with two non-zero entries glues two projectives at a $y$-type vertex, and the first and last row kill the vector spaces at the interval endpoints, $y_{i-1}=h\za_i$ and $y_j=h\zb_j$.  

The bijection on intervals is given by setting $f(J) = J'$ where $\II_{J'} \simeq \coker \Phi_{J}$; for example, we have $f([\za_i, \zb_j]) = [\zb_i, \za_j]$ from the preceding paragraph.  Since we only need to know that the bijection exists, and $f$ only appears in this proof, details of the other cases are omitted.
This bijection shows that the collection $\{\Phi_{J}\}$ is a set of projective resolutions of all indecomposables $\{\II_J\}$.

These allow us to compute all $\dim \Hom_Q(\II_J, V)$. Given $V \in \rep_Q(\bd)$, apply the functor $\Hom_Q(?, V)$ to the resolution
\begin{equation}
0 \to P^1 \xrightarrow{\Phi_{J}} P^0 \to \II_{f(J)} \to 0
\end{equation}
to get the exact sequence
\begin{equation}\label{eq:homphi}
0 \to \Hom_Q (\II_{f(J)]}, V) \to \Hom_Q (P^0, V) \xrightarrow{\Hom_Q(\Phi_{J}, V)} \Hom_Q (P^1, V).
\end{equation}
Then (again from Section \ref{sect:proj}) we can naturally identify $\Hom_Q(\Phi_{J}, V)$ with $M_J(V)$ via the diagram
\begin{equation}
\vcenter{\hbox{\begin{tikzpicture}
\node (ul) at (0,2) {$\Hom_Q (P^0, V)$};
\node (ur) at (5,2) {$\Hom_Q (P^1, V)$};
\node (ll) at (0,0) {$\bigoplus_k V_{x_k}$};
\node (lr) at (5,0) {$\bigoplus_k V_{y_k}$};
\path[->]
(ul) edge node[above] {$\Hom_Q(\Phi_{J}, V)$} (ur)
(ll) edge node[above] {$M_J(V)$} (lr)
(ul) edge node[left] {$\cong$} (ll)
(ur) edge node[right] {$\cong$} (lr);
\end{tikzpicture}}}
\end{equation}
(where the precise collection of $V_i$ in the bottom row depends on the interval type).  Then this diagram along with \eqref{eq:homphi} shows that
\begin{equation}
r_J(V) + \dim \Hom_Q(\II_{f(J)}, V) = \rank M_J(V) + \ker M_J(V) = \sum_k \bd(x_k),
\end{equation}
which gives equation \eqref{eq:rankhom} of the proposition and completes the proof.
\end{proof}

\subsection{Quiver rank arrays}
By Proposition \ref{prop:rankorbit}, an orbit is completely determined by an array of nonnegative integers. This array of numbers is the focus of this section.

\begin{definition}\label{def:satisfies}
A \textbf{quiver rank array} is a function
\begin{equation}
\br \colon \{\textup{intervals in }Q\} \to \mathbb{Z}_{\geq 0}
\end{equation}
such that there exists $V \in \rep_Q(\bd)$ with $\br_J = r_J(V)$ for all intervals $J$. 
In this case we say that $V$ \textbf{satisfies} the quiver rank array $\br$.
\end{definition}

For the remainder of the paper, we use the following notation.

\begin{notn}
Because orbits are in one-to-one correspondence with quiver rank arrays, we let $\mathcal{O}_\br$ denote the orbit determined by the quiver rank array $\br$, and we let $\overline{\mathcal{O}}_{\br}$ denote its closure. 
\end{notn}

\begin{remark}
Quiver rank arrays are partially ordered in the standard way for functions, namely, $\br' \leq \br$ when $\br'_J \leq \br_J$ for all intervals $J \subseteq Q$. At the end of the next section, we'll see that $\overline{\mathcal{O}}_{\br'}$ is contained in $\overline{\mathcal{O}}_{\br}$ if and only if $\br'\leq \br$ (see Theorem \ref{thm:mainThmSection3}).
\end{remark}

The remainder of this section is concerned with the question of when a given function $\br \colon \{\textup{intervals in }Q\} \to \mathbb{Z}_{\geq 0}$ is actually a quiver rank array. To begin, we let $\II_{J'}$ denote the indecomposable representation supported on the interval $J'$, and we let $\#J'$ denote the number of arrows in the interval $J'$. Observe that
\begin{equation}\label{eq:ranklaceeval}
   r_J (\II_{J'}) = \left\lceil \frac{\#(J \cap J')}{2} \right\rceil, 
\end{equation}
where we use $\lceil x \rceil$ to denote the least integer that is greater than or equal to $x$.
Therefore, if $V\in \rep(Q)$ is isomorphic to the direct sum of indecomposables
\[
V\cong \bigoplus_{J' \subseteq Q} s_{J'}\, \II_{J'}, \quad s_{J'} \in \mathbb{Z}_{\geq 0},\] we see that
\begin{equation}\label{eq:sumranklaceeval}
r_J(V) = \sum_{J' \subseteq Q}s_{J'}\left\lceil \frac{\#(J \cap J')}{2} \right\rceil.
\end{equation} 
This is the bipartite analogue of the \textbf{lace to rank formula} (1.2) from \cite{KMS}. Inverting the relation (\ref{eq:sumranklaceeval}) provides a way to check when a function $\br \colon \{\textup{intervals in }Q\} \to \mathbb{Z}_{\geq 0}$ is actually a quiver rank array (i.e. when there is a $V\in \rep_Q(\bd)$ that satisfies $\br$). To write down this \textbf{rank to lace formula}, we introduce some ad hoc notation (used only in Lemma \ref{lem:rankToLace}): let $J$ be an arbitrary interval of $Q$, and let $J_L$ (resp., $J_R$) denote the interval obtained by shifting $J$ one edge to the left (resp. right). If shifting the interval would take it outside of $Q$, we simply truncate it to lie within $Q$. (Alternatively, we can think of this as working with a longer quiver with dimension vector $0$ at the new vertices.)

\begin{lemma}\label{lem:rankToLace}
Let $V\in \rep_Q(\bd)$, let $J$ be an interval with at least one arrow, and let $s_J(V)$ denote the multiplicity of the indecomposable $\II_J$ in the Krull-Schmidt decomposition of $V$. Then, we have
\begin{equation}\label{eq:ranklace}
s_J(V) = (-1)^{\#J}(r_{J_L}(V) + r_{J_R}(V) - r_{J_L \cap J_R}(V) - r_{J_L\cup J_R}(V)).
\end{equation}
\end{lemma}

\begin{proof}
Notice that $r_J(V_1\oplus V_2) = r_J(V_1)+r_J(V_2)$, for any representations $V_1, V_2\in \rep(Q)$. So, it suffices to check that the righthand side of (\ref{eq:ranklace}) is $1$ when $V = \II_J$ and is $0$ when $V = \II_{J'}$ for $J'\neq J$. Furthermore, it is enough to check the values of the righthand side of (\ref{eq:ranklace}) when $V = \II_{J'}$ for $J'$ a subinterval of $J_L\cup J_R$. To this end, divide $J_L\cup J_R$ into five intervals: 
\[I_1 := J_L - (J\cap J_L),~~~ I_2:= J - (J\cap J_R), ~~~I_3:=J_L\cap J_R,\]\[I_4:=J - (J\cap J_L), ~~~I_5:=J_R - (J\cap J_R).\] 
We must check that the righthand side of (\ref{eq:ranklace}) takes the value $1$ when the leftmost arrow of $J'$ lies in interval $I_2$ and the rightmost arrow lies in interval $I_4$, and takes the value $0$ in all the other cases. We do one of these easy checks and leave the remainder to the reader.

Suppose that $V = \II_J$, so that the leftmost arrow of $J$ is in $I_2$ and the rightmost arrow is in $I_4$. If $\#J$ is even (and is nonzero), then $r_{J_L}(V) = r_{J_R}(V) = r_{J_L\cup J_R}(V) = r_J(V)$ and $r_{J_L\cap J_R}(V) = r_J(V)-1$. If $\#J$ is odd, then $r_{J_L}(V) = r_{J_R}(V) = r_{J_R\cap J_L} = r_J(V)-1$ and $r_{J_L\cup J_R}(V) = r_J(V)$. In either case, the right hand side of equation (\ref{eq:ranklace}) is $1$.
\end{proof}

\begin{remark}
Using equation (\ref{eq:ranklace}) we can recover the Krull-Schmidt decomposition of $V\in \rep_Q(\bd)$. Indeed, equation (\ref{eq:ranklace}) gives the multiplicities of the indecomposables supported on the various intervals with at least one arrow. The multiplicities of the remaining indecomposables (i.e. those supported on a single vertex of $Q$) can then be computed because the dimension vector $\bd$ is fixed. 
\end{remark}

The following corollary is now immediate.

\begin{corollary}
The function $\br \colon \{\textup{intervals in }Q\} \to \mathbb{Z}_{\geq 0}$ is a quiver rank array if and only if the quantity
\[(-1)^{\#J}(\br(J_L) + \br(J_R) - \br(J_L \cap J_R) - \br(J_L\cup J_R) )\]
is nonnegative for each interval $J\subseteq Q$.
\end{corollary}


\section{A Zelevinsky map for bipartite type $A$ quivers}


Given a quiver rank array $\br$ (see Definition \ref{def:satisfies}), let $I_{\br}$ denote the ideal in the coordinate ring $K[\rep_Q(\bd)]$ defined by \begin{equation}\label{eqn:Ibq} I_{\br}:=\big\langle \text{minors of size }(1+\br_J)\textrm{ in }M_J~\big|~ J\subseteq Q\big\rangle.\end{equation} Let $\zO_{\br} := \Spec(K[\rep_Q(\bd)]/I_\br)$ be the corresponding closed subscheme of $\rep_Q(\bd)$.  

In this section, we show that the orbit closure $\overline{\mathcal{O}}_{\br}$ is scheme-theoretically isomorphic to $\zO_{\br}$, and we show that both of these schemes are isomorphic to a Kazhdan-Lusztig variety in an opposite Schubert cell.  This gives an identification of the poset of orbit closures with a particular subset of a symmetric group under Bruhat order (see Theorem \ref{thm:mainThmSection3}).

\begin{remark}
In the terminology of a recent paper of Riedtmann and Zwara \cite{MR3008913}, $\zO_\br$ is a ``rank scheme''.   We learned after completing the first version of this article that they also have proven that $\zO_\br \simeq \overline{\mathcal{O}}_\br$ using the method of Hom-controlled functors.
\end{remark}

This section is outlined as follows: 
\begin{itemize}
\item In Section \ref{sect:qrcToNW} we define a closed immersion from each quiver representation space to an opposite Schubert cell of a partial flag variety.
  We call this the ``Zelevinsky map'' in analogy with the equioriented setting, since it also converts quiver rank arrays to rank conditions on certain northwest submatrices of the cell (cf. \cite[\S 1.3]{KMS} for the equioriented case, and also \cite{Zgradednilp, LMdegen}). 
More precisely, it identifies each $\zO_{\br}$ with a subscheme $NW_{\bb(\br)}\subseteq \mcell$ defined by the vanishing of certain minors of northwest submatrices. 
\item Section \ref{sect:zelPerm} is a combinatorial interlude; in analogy with the equioriented setting \cite[Def.~1.7]{KMS}, we define a ``Zelevinsky permutation'' $v(\br)$ associated to a quiver rank array $\br$.
\item In Section \ref{sect:OCtoKL}, we show that the subscheme $NW_{\bb(\br)}$ from Section \ref{sect:qrcToNW} is isomorphic to the Kazhdan-Lusztig variety $X_{v(\br)}\cap\cell$. This gives our main result: we see that $\zO_{\br}$ is reduced, irreducible, and isomorphic to $\overline{\mathcal{O}}_{\br}$. We also get a formula for its dimension. 
\end{itemize}


\subsection{From quiver rank arrays to northwest block rank conditions}\label{sect:qrcToNW}


Let $M_Q$ be the matrix defined in \eqref{eq:M}, so evaluating $M_Q$ at a representation 

\begin{equation}\label{eq:typearep}
V=\quad
\vcenter{\hbox{\begin{tikzpicture}[point/.style={shape=circle,fill=black,scale=.5pt,outer sep=3pt},>=latex]
   \node[outer sep=-2pt] (y0) at (-1,0) {$V_{y_0}$};
   \node[outer sep=-2pt] (x1) at (0,1) {$V_{x_1}$};
  \node[outer sep=-2pt] (y1) at (1,0) {$V_{y_1}$};
   \node[outer sep=-2pt] (x2) at (2,1) {$V_{x_2}$};
  \node[outer sep=-2pt] (y2) at (3,0) {$V_{y_2}$};
  \node (dots) at (3.75, 0.5) {$\cdots\cdots$};
   \node[outer sep=-2pt] (x3) at (5,1) {$V_{x_{n-1}}$};
  \node[outer sep=-2pt] (y3) at (6,0) {$V_{y_{n-1}}$};
  \node[outer sep=-2pt] (x4) at (7,1) {$V_{x_n}$};
  \node[outer sep=-2pt] (y5) at (8,0) {$V_{y_n}$};
  \path[->]
  	(x1) edge node[left] {$V_{\za_1}$} (y0) 
	(x1) edge node[left] {$V_{\zb_1}$} (y1)
  	(x2) edge node[left] {$V_{\za_2}$} (y1) 
	(x2) edge node[left] {$V_{\zb_2}$} (y2)
	(x3) edge node[left] {$V_{\zb_{n-1}}$} (y3)
  	(x4) edge node[left] {$V_{\za_n}$} (y3) 
  	(x4) edge node[left] {$V_{\zb_n}$} (y5); 
   \end{tikzpicture}}}
\end{equation}
gives a ``snake matrix'' $M_Q(V)$ containing all the maps in $V$.
Define the \textbf{Zelevinsky map} $\zeta$ by
\begin{equation}\label{eq:zelmap}
\zeta \colon \rep_Q(\bd) \to \mcell, \quad V \mapsto
\begin{pmatrix}
M_Q(V) & \bid_{d_y}\\
\bid_{d_x} & 0 \\
\end{pmatrix} .
\end{equation}
Expanding this out into block form is useful to see what the map does (Figure \ref{fig:bigzimage}).
\begin{figure}
\[\zeta(V) =
\begin{pmatrix}
\begin{tikzpicture}[every node/.style={minimum width=1.5em}]
\matrix (m0) [matrix of math nodes,nodes in empty cells]
{
& & & V_{\za_1}& \bid_{\bd(y_0)} &  & \\
& &  V_{\za_2} & V_{\zb_1}& & \bid_{\bd(y_1)} & \\
& \Ddots & \Ddots\\
V_{\za_n} & V_{\zb_{n-1}} &\\
V_{\zb_n} & & &&\phantom{X} &\phantom{X}&&\bid_{\bd(y_{n})} \\
\bid_{\bd(x_n)} &  & \\
& \bid_{\bd(x_{n-1})} &  & \\
\phantom{X} &  & \\
\phantom{X} &  & \\
& & &\bid_{\bd(x_{1})} \\
};
\node[scale=3] (zero) at (m0-8-2 -| m0-2-6.south east)  {$0$};
\draw[thick] (m0-6-1.north west) -- (m0-6-1.north west -| m0-5-8.south east);
\draw[thick] (m0-1-5.north west) -- (m0-1-5.north west |- m0-10-4.south east);
\draw[thick,loosely dotted] (m0-7-2) -- (m0-10-4);
\draw[thick,loosely dotted] (m0-2-6) -- (m0-5-8);
\end{tikzpicture}
\end{pmatrix}\]
\caption{Image of Zelevinsky map}\label{fig:bigzimage}
\end{figure}
\comment{
Expanding it out into the natural block form, we have
\begin{equation}
\zeta(V)= \begin{pmatrix}
\begin{tikzpicture}[every node/.style={minimum width=2em,minimum height=1em}]
\matrix (m) [matrix of math nodes, nodes in empty cells]
{
& &  V_{\za_1} &\bid_{\bd(y_0)} &  & \\
& & V_{\zb_1}  && \ddots & \\
& \Ddots&  \Ddots & \\
V_{\za_n} & \\
V_{\zb_n}&& && &&&\bid_{\bd(y_n)} \\\\
\bid_{\bd(x_n)} &  & \\
 &  & \\
 &&\\
&  & &\bid_{\bd(x_1)} \\
};
\draw (m-5-1.south west) -- (m-5-8.south east);
\end{tikzpicture}
\end{pmatrix}
\end{equation}}
This map is the closed immersion given by the homomorphism of $K$-algebras
\begin{equation}
{\zeta^*}\colon K[\mcell] \to K[\rep_Q(\bd)] 
\end{equation}
with kernel $I_\bd$, the ideal generated by setting appropriate entries of \eqref{eq:cell} to zero.

Notice that matrices in $\mcell$ are naturally partitioned into $2n+1$ blocks of rows and columns.
We label these blocks in the standard way by $1, 2, \dotsc, 2n+1$, from top to bottom and left to right.
For a matrix $Z\in \mcell$,
denote by $Z_{i\times j}$ the northwest justified submatrix of $Z$ whose southeast corner is the block in block row $i$ and block column $j$.

\begin{lemma}\label{lem:nwbrwelldef}
For any $V\in \rep_Q(\bd)$, the ranks of all $\zeta(V)_{i \times j}$ depend only on the orbit of $V$.
\end{lemma}
\begin{proof}
Using Figure \ref{fig:bigzimage}, we see how the $\GL(\bd)$-action on $\rep_Q(\bd)$ is essentially translated into row and column operations within blocks of $\mcell$.  So it does not change the ranks of block submatrices.
\end{proof}

\begin{definition}\label{def:quivertonwrank}
Let $\br$ be a quiver rank array and $V \in \mathcal{O}_\br$. The \textbf{block rank matrix associated to $\br$} is the $(2n+1)\times (2n+1)$ matrix $\bb(\br)$ defined by \[\bb(\br)_{i,j} = \rank \left(\zeta(V)_{i \times j}\right).\qedhere\]
\end{definition}

The following proposition summarizes the content of Appendix \ref{appendix}.  Its proof is separated from the main body of the paper because it is  somewhat technical and uses different notation than the rest of this section.

\begin{proposition}\label{prop:combDataEquiv}
The combinatorial data in a quiver rank array $\br$ and northwest block rank matrix $\bb(\br)$ are equivalent, in the sense that a matrix $Z \in \mcell$ satisfies
\[
\rank Z_{i \times j} = \bb(\br)_{i,j} \quad \text{for all }i,j
\]
if and only if $Z =\zeta(V)$ for some $V$ satisfying $\br$. 
\end{proposition}



The following example illustrates the essential features of how the Zelevinsky map converts a quiver rank array to a northwest block rank matrix.

\begin{example}\label{ex:quivertoblockrank}
Let $Q$ be the quiver 
\begin{equation}
Q=\quad
\vcenter{\hbox{\begin{tikzpicture}[point/.style={shape=circle,fill=black,scale=.5pt,outer sep=3pt},>=latex]
   \node[outer sep=-2pt] (y0) at (-1,0) {${y_0}$};
   \node[outer sep=-2pt] (x1) at (0,1) {${x_1}$};
  \node[outer sep=-2pt] (y1) at (1,0) {${y_1}$};
   \node[outer sep=-2pt] (x2) at (2,1) {${x_2}$};
  \node[outer sep=-2pt] (y2) at (3,0) {${y_2}$};
   \node[outer sep=-2pt] (x3) at (4,1) {${x_3}$};
  \node[outer sep=-2pt] (y3) at (5,0) {${y_3}$};
  \path[->]
  	(x1) edge node[left] {${\za_1}$} (y0) 
	(x1) edge node[left] {${\zb_1}$} (y1)
  	(x2) edge node[left] {${\za_2}$} (y1) 
	(x2) edge node[left] {${\zb_2}$} (y2)
  	(x3) edge node[left] {${\za_3}$} (y2)
	(x3) edge node[left] {${\zb_3}$} (y3);
   \end{tikzpicture}}}
\end{equation}
and let $\bd$ be a dimension vector. Then for any $V\in\rep_Q(\bd)$, we have that $\zeta(V)$ is a block matrix of the form 
\[
\left(\begin{array}{ccc|cccc}
\color{red}0&\color{red}0&V_{\za_1}&\color{blue}I_{\bd(y_0)}&\color{blue}0&\color{blue}0&\color{blue}0\\
\color{red}0&V_{\za_2}&V_{\zb_1}&\color{blue}0&\color{blue}I_{\bd(y_1)}&\color{blue}0&\color{blue}0\\
V_{\za_3}&V_{\zb_2}&\color{red}0&\color{blue}0&\color{blue}0&\color{blue}I_{\bd(y_2)}&\color{blue}0\\
V_{\zb_3}&\color{red}0&\color{red}0&\color{blue}0&\color{blue}0&\color{blue}0&\color{blue}I_{\bd(y_3)}\\
\hline
\color{blue}I_{\bd(x_3)}&\color{blue}0&\color{blue}0&\color{blue}0&\color{blue}0&\color{blue}0&\color{blue}0\\
\color{blue}0&\color{blue}I_{\bd(x_2)}&\color{blue}0&\color{blue}0&\color{blue}0&\color{blue}0&\color{blue}0\\
\color{blue}0&\color{blue}0&\color{blue}I_{\bd(x_1)}&\color{blue}0&\color{blue}0&\color{blue}0&\color{blue}0\\
\end{array}
\right) .
\]
Suppose that $\bd = (1,2,3,2,3,2,1)$. Then this determines that $\bb(\br)$ must have the form
\[
\left(\begin{array}{ccc|cccc}
\color{red}{0} & \color{red}0& *&\color{blue}1&\color{blue}1&\color{blue}1&\color{blue}1\\
\color{red}0 & * &*&*&\color{blue}4&\color{blue}4&\color{blue}4\\
* &*&*&*&*&\color{blue}7&\color{blue}7\\
* &*&*&*&*&*&\color{blue}8\\
\hline
\color{blue}2&*&*&*&*&\color{red}9&\color{blue}10\\
\color{blue}2&\color{blue}4&*&*&\color{red}8&\color{red}11&\color{blue}12\\
\color{blue}2&\color{blue}4&\color{blue}6&\color{blue}7&\color{blue}10&\color{blue}13&\color{blue}14
\end{array}\right).
\]
The blue values come from $Z$ being an element of the cell $\mcell$ (cf. ``cell conditions'' Lemma \ref{lem:cellranks}).
The red values ensure zero blocks in the northwest quadrant so that $Z$ is in the image of $\zeta$ (cf. ``image conditions'' Lemma \ref{lem:NWSE}). The entries labelled $*$ are determined by $\br$, and this data is equivalent to specifying an orbit (cf. ``orbit conditions'' Lemma \ref{lem:blockRankConditions}). For example, consider the orbit of the representation $V$ with \[V_{\alpha_1} = \begin{pmatrix}1&0\end{pmatrix},~~V_{\beta_1} = \begin{pmatrix}0&1\\0&0\\1&0\end{pmatrix},~~ V_{\alpha_2} = \begin{pmatrix}0&1\\0&0\\0&0\end{pmatrix},\] \[V_{\beta_2} =\begin{pmatrix}1&0\\0&1\\0&0\end{pmatrix},~~V_{\alpha_3} = \begin{pmatrix}0&0\\0&0\\0&0\end{pmatrix},~~ V_{\beta_3} = \begin{pmatrix}0&1\end{pmatrix}.\]
Plugging this in above, the associated block rank matrix is calculated to be
\[\bb(\br) = 
\left(\begin{array}{ccc|cccc}
\color{red}{0} & \color{red}0& 1 &\color{blue}1&\color{blue}1&\color{blue}1&\color{blue}1\\
\color{red}0 &1 &2&3&\color{blue}4&\color{blue}4&\color{blue}4\\
0 &2&4&5&6&\color{blue}7&\color{blue}7\\
1 &3&5&6&7&8&\color{blue}8\\
\hline
\color{blue}2&4&6&7&8&\color{red}9&\color{blue}10\\
\color{blue}2&\color{blue}4&6&7&\color{red}8&\color{red}11&\color{blue}12\\
\color{blue}2&\color{blue}4&\color{blue}6&\color{blue}7&\color{blue}10&\color{blue}13&\color{blue}14
\end{array}\right).
\]
The quiver rank array $\br$ can be recovered from the black entries of this matrix by the formulas of Lemma \ref{lem:blockRankConditions}.
\end{example}

Any block rank matrix naturally determines a closed subscheme of $\mcell$.

\begin{definition}\label{def:NWBlockRankVariety}
Let $\br$ be a quiver rank array and 
\[
Z = \begin{pmatrix}* & \bid_{d_y}\\\bid_{d_x} & 0 \end{pmatrix}
\]
a generic matrix of $\mcell$.   Define the \textbf{northwest block rank ideal} to be the ideal 
\[
I_{\bb(\br)}:= \langle \text{minors of size } (\bb(\br)_{i,j}+1) \textrm{ in } Z_{i \times j}~|~1\leq i, j\leq 2n+1\rangle.
\]
Define the \textbf{northwest block rank variety} to be $NW_{\bb(\br)} := \Spec \left( K[\mcell]/I_{\bb(\br)}\right)$.
\end{definition}

\begin{proposition}\label{prop:quivertopositroid}
Let $\br$ be a quiver rank array. The Zelevinsky map $\zeta$ restricts to a scheme-theoretic isomorphism from $\zO_\br$ to $NW_{\bb(\br)}$ (i.e., $(\zeta^*)^{-1}(I_\br) = I_{\bb(\br)}$). 
\end{proposition}

\begin{proof}
The proof proceeds essentially along the same lines as Appendix \ref{appendix}, but working with minors in matrices of variables instead of ranks of matrices of scalars.  It is seen there that the generators of $I_{\bb(\br)}$ coming from a block rank matrix  have 3 types.


Those coming from a ``cell condition'' entry of $\bb(\br)$ (see Lemma \ref{lem:cellranks}, or blue entries in Example \ref{ex:quivertoblockrank}) give minors of \eqref{eq:cell} which are identically zero, because the sizes of these minors are larger than the corresponding submatrix.

The ``image condition'' type entries in $\bb(\br)$ (see Lemma \ref{lem:NWSE}, or red entries in Example \ref{ex:quivertoblockrank}) gives minors that cut $NW_{\bb(\br)}$ down to lie in the scheme-theoretic image of $\zeta$.
This collection of minors 
generates the kernel $I_\bd$ of the induced map on coordinate rings.

Thus, working modulo $I_\bd$, we may assume that all remaining generators of $I_{\bb(\br)}$ are minors of 
\begin{equation}
\tilde{Z} =
\begin{pmatrix}
\begin{tikzpicture}[every node/.style={minimum width=1.5em}]
\matrix (m0) [matrix of math nodes,nodes in empty cells]
{
& & & A_1& \bid_{\bd(y_0)} &  & \\
& &  A_2 & B_1& & \bid_{\bd(y_1)} & \\
& \Ddots & \Ddots\\
A_n & B_{n-1} &\\
B_n & & &&\phantom{X} &\phantom{X}&&\bid_{\bd(y_{n})} \\
\bid_{\bd(x_n)} &  & \\
& \bid_{\bd(x_{n-1})} &  & \\
\phantom{X} &  & \\
\phantom{X} &  & \\
& & &\bid_{\bd(x_{1})} \\
};
\node[scale=3] (zero) at (m0-8-2 -| m0-2-6.south east)  {$0$};
\draw[thick] (m0-6-1.north west) -- (m0-6-1.north west -| m0-5-8.south east);
\draw[thick] (m0-1-5.north west) -- (m0-1-5.north west |- m0-10-4.south east);
\draw[thick,loosely dotted] (m0-7-2) -- (m0-10-4);
\draw[thick,loosely dotted] (m0-2-6) -- (m0-5-8);
\end{tikzpicture}
\end{pmatrix}
\end{equation}
where, as usual, there is a zero in every blank location.

The remaining generators come from the ``orbit conditions'' (see Lemma \ref{lem:blockRankConditions}, or black entries in Example \ref{ex:quivertoblockrank}). 
For each block position $(i,j)$, the lemma associates a certain interval $J \subseteq Q$, and all such $J$ arise this way.
Then one checks using the same linear algebra discussed in the proof of Lemma \ref{lem:blockRankConditions} that $\zeta^*$ induces a bijection between the prescribed size minors of $\tilde{Z}_{i\times j}$ and the minors of another size in $M_J$, which are the generators of $I_{\br}$. 
This bijection, taken over all $(i,j)$, gives that $(\zeta^*)^{-1}(I_{\br}) = I_{\bb(\br)}$.
\end{proof}


\subsection{A Zelevinsky permutation for the bipartite setting}\label{sect:zelPerm}
In this section we follow ideas similar to those in \cite{KMS} to construct the permutation $v(\br)$ representing the Kazhdan-Lusztig variety with which we ultimately identify $\zO_\br$.


\begin{proposition}\label{prop:zperm}
Let $\br$ be a quiver rank array and let $\bb(\br)$ be the associated block rank matrix of Definition \ref{def:quivertonwrank}. Then there exists a unique $d\times d$ permutation matrix $v(\br)$ that satisfies the following conditions:
\begin{enumerate}
\item the number of 1s in block $(i,j)$ of $v(\br)$ is equal to 
\[
\bb(\br)_{i,j} + \bb(\br)_{i-1,j-1} - \bb(\br)_{i,j-1} - \bb(\br)_{i-1,j}
\]
where $\bb(\br)_{i,j}=0$ if $i$ or $j$ is outside of the range $[1,2n+1]$;
\item the 1s are arranged from northwest to southeast across each block row;
\item the 1s are arranged from northwest to southeast down each block column.
\end{enumerate}
\end{proposition}

In analogy with the equioriented setting,
we call the permutation $v(\br)$ the \textbf{Zelevinsky permutation} associated to $\br$.  Before the proof, we continue Example \ref{ex:quivertoblockrank}.

\begin{example}
One constructs $v(\br)$ from $\bb(\br)$ by working from the northwest, corner filling in each block, moving either down or across.  The Zelevinsky permutation associated to the $\br$ in Example \ref{ex:quivertoblockrank} is below, where the empty blocks contain all zeros.
\begin{equation}
v(\br) = \left(\begin{array}{cc|cc|cc|c|ccc|ccc|c}
&&& & 1 & 0 &&&&&&&&\\
\hline
& & 1 & 0& & & 0 & 0 & 0 & 0 &&&&\\
& & 0 & 0& & & 1 & 0 & 0 & 0 &&&&\\
& & 0 & 0& & & 0 & 1 & 0 & 0 &&&&\\
\hline
& & 0 & 1& 0 & 0 & &&& & 0 & 0 & 0 & \\
& & 0 & 0& 0 & 1 & &&& & 0 & 0 & 0 & \\
& & 0 & 0& 0 & 0 & &&& & 1 & 0 & 0 & \\
\hline
1 & 0 & &&&&&&&&&&&\\
\hline
0 & 1 & &&&&&&&&&&& 0\\
0 & 0 & &&&&&&&&&&& 1\\
\hline
& & &&&&&&& & 0 & 1 & 0 & \\
& & &&&&&&& & 0 & 0 & 1 & \\
\hline
& & &&&& & 0 & 1 & 0 & &&&\\
& & &&&& & 0 & 0 & 1 & &&&\\
\end{array}\right)
\end{equation}
\end{example}
\comment{
v(\br) = \left(\begin{array}{cc|cc|cc|c|ccc|ccc|c}
0 & 0 & 0 & 0& 1 & 0 & 0 & 0 & 0 & 0 & 0 & 0 & 0 & 0\\
\hline
0 & 0 & 1 & 0& 0 & 0 & 0 & 0 & 0 & 0 & 0 & 0 & 0 & 0\\
0 & 0 & 0 & 0& 0 & 0 & 1 & 0 & 0 & 0 & 0 & 0 & 0 & 0\\
0 & 0 & 0 & 0& 0 & 0 & 0 & 1 & 0 & 0 & 0 & 0 & 0 & 0\\
\hline
0 & 0 & 0 & 1& 0 & 0 & 0 & 0 & 0 & 0 & 0 & 0 & 0 & 0\\
0 & 0 & 0 & 0& 0 & 1 & 0 & 0 & 0 & 0 & 0 & 0 & 0 & 0\\
0 & 0 & 0 & 0& 0 & 0 & 0 & 0 & 0 & 0 & 1 & 0 & 0 & 0\\
\hline
1 & 0 & 0 & 0& 0 & 0 & 0 & 0 & 0 & 0 & 0 & 0 & 0 & 0\\
\hline
0 & 1 & 0 & 0& 0 & 0 & 0 & 0 & 0 & 0 & 0 & 0 & 0 & 0\\
0 & 0 & 0 & 0& 0 & 0 & 0 & 0 & 0 & 0 & 0 & 0 & 0 & 1\\
\hline
0 & 0 & 0 & 0& 0 & 0 & 0 & 0 & 0 & 0 & 0 & 1 & 0 & 0\\
0 & 0 & 0 & 0& 0 & 0 & 0 & 0 & 0 & 0 & 0 & 0 & 1 & 0\\
\hline
0 & 0 & 0 & 0& 0 & 0 & 0 & 0 & 1 & 0 & 0 & 0 & 0 & 0\\
0 & 0 & 0 & 0& 0 & 0 & 0 & 0 & 0 & 1 & 0 & 0 & 0 & 0\\
\end{array}\right) 
}

\begin{proof}[Proof of Proposition \ref{prop:zperm}]
We must show that the number of $1$s in a given block row (respectively block column) is equal to the height of that block row (resp. width of that block column). This is all that is needed to prove the proposition, since conditions (2) and (3) determine a unique arrangement of the $1$s within each block. 

From condition (1), the number of $1$s in block row $i$ is 
\[
\begin{split}
& \sum_{j=1}^{2n+1}\left(\bb(\br)_{i,j}+\bb(\br)_{i-1,j-1}-\bb(\br)_{i,j-1}-\bb(\br)_{i-1,j}\right)=\\
&\bb(\br)_{i,2n+1}-\bb(\br)_{i-1,2n+1} = \text{height of block } i.
\end{split}
\]
The computation for columns is completely analogous.
\end{proof}

We record some useful properties of the Zelevinsky permutation in the following lemma.

\begin{lemma}\label{cor:essBox}\label{lem:length}\label{lem:minLength}
For any quiver rank array $\br$, the following hold.
\begin{enumerate}
\item The Zelevinsky permutation $v(\br)$ is the minimal length element in its $(W_P,W_P)$-double coset. 

\item Every essential box in $\mathcal{E}ss(v(\br))$ occurs in the southeast corner of a block. 

\item The Zelevinsky permutation $v(\br)$ has length \[\displaystyle\sum_{i=2}^{2n+1}\displaystyle\sum_{j=1}^{2n}(\bb(\br)_{i-1,n+1}-\bb(\br)_{i-1,j})(\bb(\br)_{i,j}+\bb(\br)_{i-1,j-1}-\bb(\br)_{i,j-1}-\bb(\br)_{i-1,j}).\]
\end{enumerate}
\end{lemma}

\begin{proof}
Parts (1) and (2) follow immediately from the definition of $v(\br)$, since these happen precisely when the $1$s in the corresponding permutation matrix appear northwest to southeast in each block row and column.

For (3), observe that, given a permutation $v\in S_d$, the length of $v$ can be read off from the associated permutation matrix. The length of $v$ is the number of pairs of $1$s with the property that one of the $1$s appears northeast of the other. Thus, the length $l(v(\br))$ is
\[
\displaystyle\sum_{i=2}^{2n+1}\displaystyle\sum_{j=1}^{2n}\left(\#1s\textrm{ strictly NE of block }(i, j)\right)(\#1s\textrm{ in block }(i,j)),
\]
which gives the stated formula.
\end{proof}

\subsection{Identifying orbit closures with Kazhdan-Lusztig varieties}\label{sect:OCtoKL}
Recall from Section \ref{sect:schubert} that $v(\br)$ defines a closed subscheme of $Y_{v(\br)} \subseteq \mcell$, which is isomorphic to the Kazhdan-Lusztig variety $X_{v(\br)}\cap \cell$ in $P\backslash G$.

\begin{proposition}\label{prop:lemmaToMainThm}
For each quiver rank array $\br$, the Zelevinsky map restricts to an isomorphism from $\zO_\br$ to $Y_{v(\br)}$. Consequently, each $\zO_\br$ is reduced and irreducible.
\end{proposition}

\begin{proof}
By Proposition \ref{prop:quivertopositroid}, the Zelevinsky map restricts to an isomorphism from $\zO_{\br}$ to $NW_{\bb(\br)} := \Spec K[\mcell]/I_{\bb(\br)}$, the northwest block rank variety (see Definition \ref{def:NWBlockRankVariety}). So, to prove the proposition, it suffices to show that $I_{\bb(\br)}$ is equal to the ideal $I_{v(\br)}$ that scheme-theoretically defines $Y_{v(\br)}$. 

By construction, the rank of each northwest block submatrix $v(\br)_{i\times j }$ is equal to $\bb(\br)_{i,j}$. Thus, $I_{\bb(\br)}\subseteq I_{v(\br)}$. The reverse inclusion follows from the fact that the essential boxes of $v(\br)$ lie in the southeast corner of blocks (Lemma \ref{cor:essBox}). This guarantees that the only minors needed to generate $I_{v(\br)}$ come from \emph{block} rank conditions. Therefore, by \cite[\S3]{Fulton}, we have $I_{v(\br)}\subseteq I_{\bb(\br)}$.
So $\zeta$ restricts to a scheme-theoretic isomorphism 
\begin{equation}
\zO_\br \xrightarrow{\zeta} Y_{v(\br)} \simeq X_{v(\br)}\cap \cell,
\end{equation}
giving the last statement of the theorem by known properties of Kazhdan-Lusztig varieties.
\end{proof}

The following theorem is the main result of this section.

\begin{theorem}\label{thm:mainThmSection3}
\begin{enumerate}
\item The orbit closure $\overline{\mathcal{O}}_{\br}$ is scheme-theoretically isomorphic to $\zO_{\br}$, so the Zelevinsky map identifies orbit closures in $\rep_Q(\bd)$ with Kazhdan-Lusztig varieties in $P \backslash G$.
\item For two quiver rank arrays $\br$ and $\br'$, we have 
\[
\overline{\mathcal{O}}_{\br'}\subseteq \overline{\mathcal{O}}_{\br} \Longleftrightarrow \br' \leq \br \Longleftrightarrow v(\br') \geq v(\br).
\]
\item Let $v(\rep_Q(\bd))$ be the Zelevinsky permutation associated to the maximal quiver rank array (i.e., the dense orbit). Then, the poset of orbit closures in $\rep_Q(\bd)$, partially ordered by inclusion, is anti-isomorphic with the subposet of the symmetric group $S_d$, under Bruhat order, consisting of permutations $\pi$ satisfying both (i) $\bw\leq \pi\leq v(\rep_Q(\bd))$, and (ii) $\pi$ is a minimal length $(W_P,W_P)$-double coset representative in $S_d$.
\end{enumerate}
\end{theorem}

\begin{proof}
By definition, the $\GL(\bd)$-invariant variety $\zO_{\br}$ contains precisely those orbits $\mathcal{O}_{\br'}$ such that $\br' \leq \br$. Because each $\zO_{\br}$ is reduced (by Proposition \ref{prop:lemmaToMainThm}), we see that $\zO_{\br'}$ is a subscheme of $\zO_{\br}$ if and only if $\br' \leq \br$. Furthermore, $\mathcal{O}_{\br}$ is the only orbit in $\zO_{\br}$ that is not also contained in some lower-dimensional subvariety (namely some other $\zO_{\br'}$). Thus $\mathcal{O}_{\br}$ is dense in $\zO_{\br}$, and $\overline{\mathcal{O}_{\br}}\cong \zO_{\br}$ (again, since $\zO_{\br}$ is reduced). This proves (1) and the leftmost equivalence in (2). The rest of (2) follows from the Zelevinsky map identification of $\zO_{\br}$ with $Y_{v(\br)}$ in Proposition \ref{prop:lemmaToMainThm}. 

To prove (3), first notice that the poset of orbit closures can at least be identified with a subposet of those $\pi$ that satisfy both (i) and (ii). Indeed, each Zelevinsky permutation $v(\br)$ satisfies (ii) by definition, and satisfies condition (i) because $Y_{v(\br)}$ is a non-empty subvariety of $\zeta(\rep_Q(\bd))$ (by Proposition \ref{prop:lemmaToMainThm}). 
On the other hand, suppose that $\pi\in S_d$ satisfies both conditions (i) and (ii). By (i), $Y_{\pi}$ is a non-empty subvariety of $\zeta(\rep_Q(\bd))$ and, by (ii), all essential boxes of $\pi$ occur in the southeast corner of a block. Thus, $Y_{\pi}$ is the northwest block rank variety $NW_{\bb(\br)}$ for the quiver rank array $\br$ associated to $\zeta^{-1}(M)\in \rep_Q(\bd)$, for any $M\in Y_{\pi}$ (the choice of $M$ doesn't matter by Proposition \ref{prop:combDataEquiv}). Applying Proposition \ref{prop:quivertopositroid}, and item (1) then completes the proof. 
\end{proof}

This gives a formula for dimensions of an orbit closure in terms of its block rank matrix.
\begin{corollary}
The orbit closure $\overline{\mathcal{O}}_{\br}$ has dimension \[d_xd_y-\displaystyle\sum_{i=2}^{2n+1}\displaystyle\sum_{j=1}^{2n}(\bb(\br)_{i-1,n+1}-\bb(\br)_{i-1,j})(\bb(\br)_{i,j}+\bb(\br)_{i-1,j-1}-\bb(\br)_{i,j-1}-\bb(\br)_{i-1,j}).\] 
\end{corollary}

\begin{proof}
Because $\overline{\mathcal{O}}_{\br}$ is isomorphic to the Kazhdan-Lusztig variety $X_{v(\br)}\cap \cell$, it has dimension $l(\bw)-l(v(\br))$. Applying Lemma \ref{lem:length} yields the desired result.
\end{proof}


\section{Arbitrary orientations in type $A$}
Let $Q$ be a type $A$ quiver with arbitrary orientation. Choose one end of the quiver and go across labeling the vertices $z_0, z_1, \dotsc, z_n$ and the arrows $\zg_1, \zg_2, \dotsc, \zg_n$. We will define a bipartite, type $A$ quiver $\tq$ which has orbit closures with geometry closely connected to that of $Q$.  
The approach we take is that of ``shrinking bijective arrows", following Bongartz \cite[\S5.2]{MR1303228}.
We present a representative example before giving the precise definition.

\begin{example}
Let $Q$ be the quiver:
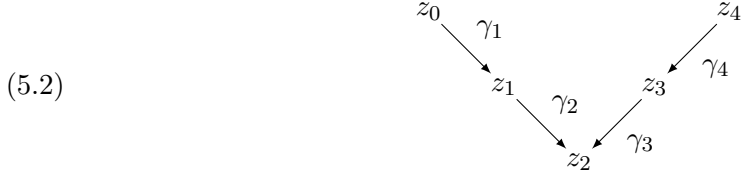
\begin{equation}\label{eq:arbitraryq}
\vcenter{\hbox{\begin{tikzpicture}[point/.style={shape=circle,fill=black,scale=.5pt,outer sep=3pt},>=latex]
   \node[outer sep=-2pt] (z0) at (0,2) {$z_0$};
   \node[outer sep=-2pt] (z1) at (1,1) {$z_1$};
  \node[outer sep=-2pt] (z2) at (2,0) {$z_2$};
   \node[outer sep=-2pt] (z3) at (3,1) {$z_3$};
  \node[outer sep=-2pt] (z4) at (4,2) {$z_4$};
  \path[->]
  	(z0) edge node[auto] {$\zg_1$} (z1) 
	(z1) edge node[auto] {$\zg_2$} (z2)
  	(z3) edge node[auto] {$\zg_3$} (z2) 
	(z4) edge node[auto] {$\zg_4$} (z3);
   \end{tikzpicture}}}
\end{equation}
Then $\tq$ will contain two new vertices $w_1, w_3$, and two new arrows $\zd_1, \zd_3$.
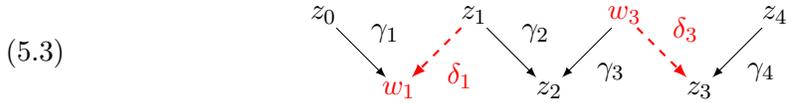
\begin{equation}\label{eq:qcover}
\vcenter{\hbox{\begin{tikzpicture}[point/.style={shape=circle,fill=black,scale=.5pt,outer sep=3pt},>=latex]
   \node[outer sep=-2pt] (z0) at (-1,1) {$z_0$};
   \node[outer sep=-2pt] (w1) at (0,0) {$\color{red}{w_1}$};
   \node[outer sep=-2pt] (z1) at (1,1) {$z_1$};
  \node[outer sep=-2pt] (z2) at (2,0) {$z_2$};
   \node[outer sep=-2pt] (w3) at (3,1) {$\color{red}{w_3}$};
  \node[outer sep=-2pt] (z3) at (4,0) {$z_3$};
  \node[outer sep=-2pt] (z4) at (5,1) {$z_4$};
  \path[->]
  	(z0) edge node[auto] {$\zg_1$} (w1) 
  	(z1) edge[thick,dashed,red] node[auto] {$\zd_1$} (w1) 
	(z1) edge node[auto] {$\zg_2$} (z2)
  	(w3) edge node[auto] {$\zg_3$} (z2) 
	(w3) edge[thick, dashed,red] node[auto] {$\zd_3$} (z3)
	(z4) edge node[auto] {$\zg_4$} (z3);
   \end{tikzpicture}}}. \qedhere
\end{equation}
\end{example}

In general, $\tq$ is defined by the following local insertions of vertices and arrows:  for each \emph{intermediate vertex} of the form $z_{i-1} \xrightarrow{\zg_{i}} z_i \xrightarrow{\zg_{i+1}}$, add an intermediate sink $w_i$ and arrow $\zd_i$ in the configuration
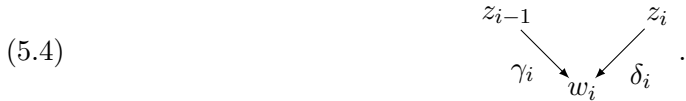
\begin{equation}\label{eq:sink}
\vcenter{\hbox{\begin{tikzpicture}[point/.style={shape=circle,fill=black,scale=.5pt,outer sep=3pt},>=latex]
   \node[outer sep=-2pt] (z0) at (-1,1) {$z_{i-1}$};
   \node[outer sep=-2pt] (w1) at (0,0) {$w_i$};
   \node[outer sep=-2pt] (z1) at (1,1) {$z_i$};
  \path[->]
  	(z0) edge node[below left] {$\zg_i$} (w1) 
  	(z1) edge node[auto] {$\zd_i$} (w1);
	   \end{tikzpicture}}}.
\end{equation}
For each {intermediate vertex} of the form $z_{i-1} \xleftarrow{\zg_{i}} z_i \xleftarrow{\zg_{i+1}}$, add an intermediate source $w_i$ and arrow $\zd_i$ in the configuration
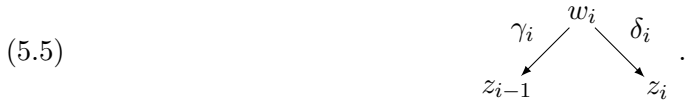
\begin{equation}\label{eq:source}
\vcenter{\hbox{\begin{tikzpicture}[point/.style={shape=circle,fill=black,scale=.5pt,outer sep=3pt},>=latex]
   \node[outer sep=-2pt] (z0) at (-1,0) {$z_{i-1}$};
   \node[outer sep=-2pt] (w1) at (0,1) {$w_i$};
   \node[outer sep=-2pt] (z1) at (1,0) {$z_i$};
  \path[->]
  	(w1) edge node[above left] {$\zg_i$} (z0) 
  	(w1) edge node[auto] {$\zd_i$} (z1);
	   \end{tikzpicture}}}.
\end{equation}

For a dimension vector $\bd$ for $Q$, define $\td$ as the natural lifting \[\td(z_i) := \bd(z_i)\textrm{ and }\td(w_i) := \bd(z_i).\]
Let $G^*=\prod\GL_{\td(w_i)}(K)$ be the base change group at the added vertices, so that 
\begin{equation}\label{eq:GLdecomp}
\GL(\td) = G^* \times \GL(\bd) .
\end{equation}
Throughout this section, we denote a typical element of $\rep_{\tq}(\td)$ by $\tv=(V_{\zd_i}) \times( V_{\zg_i})$, and an element of $\GL(\td)$ by $\tilde{g}=(g_{w_i})\times (g_{z_i})$.

\begin{proposition}\label{prop:equivariantprojection}
Let $Q$ be a quiver of type $A$, and $\tq$ the associated bipartite quiver defined above.  Then there is a $\GL(\td)$-stable open set $U \subset \rep_{\tilde{Q}}(\tilde{\bd})$ and a morphism 
\begin{equation}
\pi \colon U \to \rep_Q(\bd)
\end{equation}
which is equivariant with respect to the natural projection of base change groups 
\[\GL(\td) \to \GL(\bd),\] and also a principal $G^*$-bundle.
\end{proposition}
\begin{proof}
Let $U$ be the $\GL(\td)$-stable open set where the map over each $\zd_i$ is an isomorphism, so that $U \simeq G^* \times \rep_Q(\bd)$ as an algebraic variety.  Since the action of $\GL(\td)= G^* \times \GL(\bd)$ on $U$ is not just the factor-wise one, we have to incorporate a slight twist into $\pi$ to get equivariance.

For $\tv \in U$,  define matrices $X_{\zg_i} = V_{\zg_i}$ when $z_{i-1} \xrightarrow{\zg_i} z_i$ or $z_{i-1} \xleftarrow{\zg_{i}} z_{i}$.  We set $X_{\zg_i} = V_{\zd_i}^{-1} V_{\zg_i}$ or  $X_{\zg_i}=V_{\zg_i}V_{\zd_i}^{-1}$ when $\zg_i$ is involved in a local configuration of type \eqref{eq:sink} or \eqref{eq:source}, respectively.
Then define the projection map by
\begin{equation}
\pi \colon U \to \rep_Q(\bd), \qquad \tv \mapsto (X_{\zg_i}),
\end{equation}
which we will check is equivariant with respect to the natural projection $\GL(\td) \to \GL(\bd)$. 

Let $\tilde{g} \in \GL(\td)$ and $\tv \in U$.
For arrows in $\tq$ of the form $z_{i-1} \xrightarrow{\zg_{i}} z_{i}$ or $z_{i-1} \xleftarrow{\zg_{i}} z_{i}$ in $\tilde{Q}$, it is straightforward to see that the factor of $\tilde{g}\cdot \tv$ indexed by $\zg_i$ is either $g_{z_{i}}V_{\zg_i} g_{z_{i-1}}^{-1}$ or $g_{z_{i-1}}V_{\zg_i} g_{z_{i}}^{-1}$, which agrees with that factor in $(g_{z_i})\cdot \pi(\tv)$.  The remaining arrows are involved in a local configuration of type \eqref{eq:sink} or \eqref{eq:source}; we just write out the check for the first type because the second follows \emph{mutatis mutandis}.
Over $\tq$, the action of $\tilde{g}$ sends the pair $(V_{\zg_i}, V_{\zd_i})$ to $(g_{w_i} V_{\zg_i} g_{z_{i-1}}^{-1},\ g_{w_i} V_{\zd_i} g_{z_i}^{-1})$.  Then $\pi$ collapses this pair to $g_{z_i} V_{\zd_i}^{-1} V_{\zg_i} g_{z_{i-1}}^{-1}$ in the factor indexed by $\zg_i$, which agrees with that factor in $(g_{z_i}) \cdot \pi(\tv)$.  So $\pi$ is equivariant.

The equivariance of $\pi$ with respect to projection $G^* \times \GL(\bd) \to \GL(\bd)$ implies that the factor $G^*$ acts on fibers of $\pi$. So we just need to see that the action is free and transitive on fibers to conclude that we have a principal $G^*$-bundle.  But this is clear because each fiber of $\pi$ can be identified with $G^*$ with the action of the factor $G^* \times \{1\}\subset \GL(\td)$ by left multiplication.
\end{proof}

\begin{theorem}\label{thm:arbitraryorient}
The projection $\pi$ from Proposition \ref{prop:equivariantprojection} gives a bijection between orbits in $U$ and orbits in $\rep_Q(\bd)$; the same is true for orbit closures.  Consequently, each orbit closure $\overline{\mathcal{O}} \subseteq \rep_Q(\bd)$ for an arbitrary type $A$ quiver is isomorphic to an orbit closure of $\rep_{\tq}(\td)$ of a bipartite quiver, up to a smooth factor. Namely, we have
\begin{equation}
\overline{\pi^{-1}(\mathcal{O})} \simeq G^* \times \overline{\mathcal{O}}.
\end{equation}
\end{theorem}
\begin{proof}
Equivariance gives that orbits go to orbits, and transitivity of the $G^* \subset \GL(\td)$ action on fibers gives that there is only one $\GL(\td)$-orbit mapping to each orbit in $\rep_Q(\bd)$.  This extends by continuity to a bijection on orbit closures.  The definition of $\pi$ gives the decomposition in the last statement.
\end{proof}

\section{Consequences for the geometry of orbit closures of type $A$ quivers}

Let $Q$ be a quiver of type $A$ with arbitrary orientation and let $\bd$ be a dimension vector for $Q$. In this section, we use our previous work to recover the results of Bobi\'nski and Zwara that orbit closures in $\rep_Q(\bd)$ are normal, Cohen-Macaulay, and have rational singularities (see \cite[Thm.~1.1]{BZ-typeA}). In addition, we show that orbit closures in a fixed $\rep_Q(\bd)$ are all simultaneously compatibly Frobenius split.  

Let $\overline{\mathcal{O}}\subseteq\rep_Q(\bd)$ be an orbit closure. By Theorem \ref{thm:mainThmSection3} and Proposition \ref{thm:arbitraryorient}, there is a product of general linear groups $G^*$ with the property that $\overline{\mathcal{O}}\times G^*$ is isomorphic to an open subset of a Kazhdan-Lusztig variety. 

\begin{proposition}(compare with \cite[Thm.~1.1]{BZ-typeA})
Orbit closures in $\rep_Q(\bd)$ are normal and Cohen-Macaulay.
\end{proposition}

\begin{proof}
Let $\overline{\mathcal{O}}$ be an orbit closure. Because Kazhdan-Lusztig varieties are normal and Cohen-Macaulay (see \cite{Brion-Pos}), so is $\overline{\mathcal{O}}\times G^*$. Therefore, $\overline{\mathcal{O}}$ is also normal and Cohen-Macaulay (see, for example, \cite[\S23]{Matsumura} on flat morphisms).
\end{proof}



Recall that a variety $X$ defined over a field of characteristic $0$ has rational singularities if it is normal, and if there exists a non-singular variety $Y$ along with a proper, birational morphism $f:Y\rightarrow X$ satisfying
\begin{equation}\label{eq:ratsings}
R^if_*\mathcal{O}_Y = 0.
\end{equation}
Note that if one resolution of singularities of $X$ satisfies equation (\ref{eq:ratsings}), then all do. (See, for example, \cite[Page 50]{ToroidalEmbeddings} for further information.) 

\begin{proposition}(compare with \cite[Thm.~1.1]{BZ-typeA})
If the ground field $K$ has characteristic $0$, then orbit closures in $\rep_Q(\bd)$ have rational singularities.
\end{proposition}

\begin{proof}
Let $X \subseteq\rep_Q(\bd)$ be an orbit closure and let $f:Y\rightarrow X$ be a resolution of singularities. Since $X$ is normal, we need only show that 
\[
R^if_*\mathcal{O}_Y = 0,\quad \forall i>0.
\]
Because $X$ is affine, this is equivalent to showing that $H^i(Y,\mathcal{O}_Y) = 0$ for all $i>0$ (see \cite[III.8.5]{Hartshorne}).

Now, let $G^*$ be a product of general linear groups, chosen as in Proposition \ref{thm:arbitraryorient}, so that $X\times G^*$ is isomorphic to an open subset of a Kazhdan-Lusztig variety. Then $f\times Id: Y\times G^*\rightarrow X\times G^*$ is a resolution of singularities. 
Because $X\times G^*$ is isomorphic to an open subset of a Kazhdan-Lusztig variety, $X\times G^*$ has rational singularities (see \cite{Brion-Pos}). Therefore, since $X\times G^*$ is affine, we have
\[H^i(Y\times G^*,\mathcal{O}_Y\boxtimes\mathcal{O}_{G^*}) = 0,\quad \forall i>0.\] 
Applying the K\"unneth formula for sheaf cohomology (see \cite[Proposition 9.2.4]{Kempf}), we see that 
\begin{equation}\label{eq:kun}
\bigoplus_{i_1+i_2= i}H^{i_1}(Y,\mathcal{O}_Y)\otimes H^{i_2}(G^*,\mathcal{O}_{G^*}) = 0,\quad  \forall i>0.
\end{equation}
Since $G^*$ is affine, we have that $H^i(G^*,\mathcal{O}_{G^*}) = 0$ for all $i>0$, and so (\ref{eq:kun}) becomes \[H^i(Y,\mathcal{O}_Y)\otimes H^0(G^*,\mathcal{O}_{G^*}) = 0,\quad \forall i>0,\] which ensures that $H^i(Y,\mathcal{O}_Y) = 0$ for all $i>0$. 
\end{proof}

For the remainder of the section, suppose that $K$ is a perfect field of characteristic $p>0$. We now show that orbit closures in a fixed $\rep_Q(\bd)$ are all simultaneously compatibly Frobenius split.

Recall that a $K$-algebra $R$ (or, equivalently, $\Spec R$) is \textbf{Frobenius split} if there exists an additive map $\phi:R\rightarrow R$, satisfying both $\phi(a^pb) = a\phi(b)$, $\forall a,b\in R$, and $\phi(1) = 1$. An ideal $I\subseteq R$ is \textbf{compatibly split} by $\phi: R\rightarrow R$ if $\phi(I)\subseteq I$. Notice that if $I$ is compatibly split, then $\phi:R\rightarrow R$ descends to a Frobenius splitting of $R/I$. These definitions sheafify, and we may talk about Frobenius split schemes, and their compatibly split subschemes. See \cite[Chapter 1]{BK} for the basics for Frobenius splitting.  

\begin{proposition}\label{Fsplitting}
If the ground field $K$ is perfect of characteristic $p>0$, then there exists a Frobenius splitting $\phi:\rep_Q(\bd)\rightarrow \rep_Q(\bd)$ that simultaneously compatibly splits all orbit closures.
\end{proposition}
  
We thank Karen E. Smith for showing us how to go from the bipartite case to the general case in the proof that follows.

\begin{proof}
There is a Frobenius splitting of $P\backslash G$ which compatibly splits all Richardson varieties\footnote{In fact, by \cite{MR2588137} or \cite{KLS-Richardson}, there is a splitting of $P\backslash G$ for which the collection of compatibly split subvarieties is the set of projected Richardson varieties (i.e. the set of images of Richardson varieties in $B\backslash G$ under the projection $\pi:B\backslash G \rightarrow P\backslash G$).} (see \cite[Chapter 2]{BK}, also \cite{KLS-Richardson}). In particular, the opposite Schubert variety $X^{\bw}$, for $\bw$ as in (\ref{eq:permCell}), has an induced Frobenius splitting, and so the opposite Schubert cell $\cell$ does as well (since it is an open subvariety of $X^{\bw}$; see \cite[Lemma 1.1.7]{BK}). Notice that the Kazhdan-Lusztig varieties of the form $X_{v}\cap\cell$, $v\in W^P$ and $l(v)<l(\bw)$, are a subset of all compatibly split subvarieties of $\cell$. Applying Theorem \ref{thm:mainThmSection3} then yields the desired result in the bipartite type $A$ setting.

Next suppose that $Q$ is a type $A$ quiver with arbitrary orientation. Fix a dimension vector $\bd$ and let $G^*$ be the product of general linear groups as in Proposition \ref{thm:arbitraryorient}. Since each $\overline{\mathcal{O}}\times G^*$ is isomorphic to an open subset of a Kazhdan-Lusztig variety of the form $X_v\cap \cell$, there is a Frobenius splitting \[\phi: K[\rep_Q(\bd)]\otimes_K K[G^*]\rightarrow K[\rep_Q(\bd)]\otimes_K K[G^*]\] for which all $\overline{\mathcal{O}}\times G^*$ (among other subvarieties) are compatibly split. We have: 
\begin{equation}
\vcenter{\hbox{\begin{tikzpicture}
\node (ul) at (0,2) {$K[\rep_Q(\bd)]$};
\node (ur) at (5,2) {$K[\rep_Q(\bd)]\otimes_K K[G^*]$};
\node (ll) at (0,0) {$K[\rep_Q(\bd)]$};
\node (lr) at (5,0) {$K[\rep_Q(\bd)]\otimes_K K[G^*]$};
\path[right hook-latex]
(ul) edge node[above] {$i$} (ur);
\path[-latex]
(lr) edge node[above] {$\pi$} (ll);
\path[-latex]
(ur) edge node[right] {$\phi$} (lr);
\end{tikzpicture}}}
\end{equation}
where $i$ denotes the map $r\mapsto r\otimes 1_{K[G^*]}$, and $\pi$ denotes the map $r\otimes s\mapsto s(g_0)r$, for a fixed $g_0 \in G^*$. An easy check shows that the composition of the three maps in the diagram is a Frobenius splitting of $\rep_Q(\bd)$, and that this composition restricts to a Frobenius splitting of each orbit closure $\overline{\mathcal{O}}$ (since $\phi$ restricts to a Frobenius splitting of each $\overline{\mathcal{O}}\times G^*$).  
\end{proof}

\begin{remark}
The actual chronology of this work is in some sense the opposite of the final presentation: it began with computing examples of Frobenius splittings of quiver loci, which revealed the form of our bipartite Zelevinsky map.
\end{remark}

\appendix
\section{Converting quiver ranks to northwest ranks}\label{appendix}
We retain the notation of the main body of the article, with one exception: for matrices in $\mcell$, we label the block rows and columns by vertices of $Q_0$ in the natural way suggested by the positions of the identity matrices.  Namely, the columns are labeled $x_n, \dotsc, x_1, y_0, y_1, \dotsc, y_n$ from left to right, while the rows are labeled $y_0,  \dotsc, y_n, x_n, \dotsc, x_1$ from top to bottom.
In this section, for two vertices $v, v'$ of $Q$, we denote by $Z_{v\times v'}$ the northwest justified submatrix of $Z$ whose southeast corner is the block in block row $v$ and block columns $v'$.

\begin{lemma}[Cell conditions]\label{lem:cellranks}
For any $Z \in \mcell$, the following northwest  block submatrices automatically have maximal rank:
\begin{itemize}
\item for pairs $0\leq i \leq j \leq n$, $\rank Z_{x_i \times x_j} 
= \sum_{k=0}^{i}\bd(y_k)$;
\item for pairs $1\leq i \leq j \leq n$, $\rank Z_{y_i \times y_j} 
 = \sum_{k=j}^n\bd(x_k)$;
\item for $0\leq j\leq n$, $\rank Z_{x_1 \times y_j} 
= d_x+\sum_{k=0}^{j}\bd(y_k)$;
\item for $1\leq i\leq n$, $\rank Z_{x_i \times y_n} 
= d_y+\sum_{k=j}^n\bd(x_k)$.
\end{itemize}
\end{lemma}
\begin{proof}
This is clear from inspecting Figure \ref{fig:bigzimage}.
\end{proof}

\begin{lemma}[Image conditions]\label{lem:NWSE}
A matrix $Z\in \mcell$ is in the image of $\zeta$ if and only if both of the following conditions hold:
\begin{itemize}
\item[(NW)] for $0\leq i\leq n-2$ and $i+2\leq j\leq n$, 
\[\rank Z_{{y_i}\times {x_j}}= 0;\] 
\item[(SE)] for $2\leq i\leq n$ and $i-1\leq j\leq n-1$, 
\[\rank Z_{{x_i}\times {y_j}}=  \sum_{k=i}^n \bd(x_k)+ \sum_{k=0}^{j} \bd(y_k).\]
\end{itemize}
\end{lemma}

\begin{proof}
By definition, $Z$ is in the image of  $\zeta$ if and only if the submatrix $Z_{{y_n}\times {x_1}}$ has the ``snake'' form $M_Q(V)$ seen in Figure \ref{fig:bigzimage}. 
Condition (NW) obviously corresponds to the zeros in the northwest of $M_Q(V)$.  To get the zero entries in the southeast, consider a northwest justified submatrix $Z_{{x_i}\times {y_j}}$ from condition (SE), as seen in Figure \ref{fig:SE}.
\begin{figure}
\[
Z_{{x_i}\times {y_j}}=
\begin{pmatrix}
\begin{tikzpicture}[every node/.style={minimum width=1.5em}]
\matrix (m) [matrix of math nodes,nodes in empty cells]
{
& & & &\bid_{\bd(y_0)} &  & \\
& & & & & \ddots & \\
& & & & &  &\bid_{\bd(y_{j})} \\
& & & \clubsuit \\
\bid_{\bd(x_n)} &  & \\
 & \ddots & \\
&  &\bid_{\bd(x_{i})} \\
};
\node[scale=4] (star) at (m-5-2 |- m-2-5) {*};
\node[scale=3] (zero) at (m-6-2 -| m-2-6)  {$0$};
\draw[thick] (m-5-1.north west) -- (m-5-1.north west -| m-3-7.south east);
\draw[thick] (m-1-5.north west) -- (m-1-5.north west |- m-7-3.south east);
\draw[dotted,thick] (m-7-3.south east) -- (m-4-4.north west);
\draw[dotted,thick] (m-4-4.north west) -- (m-3-7.south east);
\end{tikzpicture}
\end{pmatrix} .
\]
\caption{Condition (SE)}\label{fig:SE}
\end{figure}
By clearing rows and columns, we see that this matrix has rank precisely $\sum_{k=i}^n \bd(x_k)+ \sum_{k=0}^{j} \bd(y_k)$ if and only if all entries is the region marked $\clubsuit$ are zero.  By varying $i$ between $2$ and $n$ and $j$ between $i-1$ and $n-1$, we get all of the blocks of zeros in the lower part of $M_Q(V)$.
\end{proof}

Finally, given that a matrix satisfies the cell and image conditions of the previous two lemmas, we record how $\zeta$ translates quiver rank conditions to northwest block rank conditions.

\begin{lemma}[Orbit conditions]\label{lem:blockRankConditions}
A representation $V\in \rep_Q(\bd)$ satisfies $\br$ if and only if $\zeta(V)$ satisfies the conditions:
\begin{itemize}
\item[(I1)] for $2\leq i\leq n$ and $1\leq j\leq i-1$,
\[
\rank \zeta(V)_{{x_i}\times {x_j}} =  \br_{[\alpha_j, \beta_{i-1}]} + \sum_{k=i}^{n} \bd(x_{k}),
\] 
and for $1\leq j\leq n$, 
\[\rank \zeta(V)_{{y_n}\times {x_j}} = \br_{[\alpha_j,\beta_{n}]};\]
\item[(I2)] for $2\leq i\leq n$ and $0\leq j\leq i-2$,
\[
\rank \zeta(V)_{{x_i}\times {y_j}} =  \br_{[\beta_{j+1}, \beta_{i-1}]} + \sum_{k=i}^{n} \bd(x_{k}) + \sum_{k=0}^j \bd(y_k),
\] 
and for $0\leq j\leq n-1$, \[\rank \zeta(V)_{{y_n}\times {x_j}} = \br_{[\beta_{j+1},\beta_n]}+\sum_{k=0}^j\bd(y_k); \]
\item[(I3)] for $0\leq i\leq n-1$ and $1\leq j\leq i+1$,
\[
\rank \zeta(V)_{{y_i}\times {x_j}} =  \br_{[\alpha_j, \alpha_{i+1}]};
\] 
\item[(I4)] for $1\leq i\leq n-1$ and $0\leq j\leq i-1$,
\[
\rank \zeta(V)_{{y_i}\times {y_j}} =  \br_{[\beta_{j+1}, \alpha_{i+1}]}+ \sum_{k=0}^j \bd(y_k).
\] 
\end{itemize}
\end{lemma}

\begin{proof}
Recall that ``$V$ satisfies $\br$'' means that $\rank M_J(V) = \br_J$ for all intervals $J \subseteq Q$.  There are four types of intervals, depending on the type of the first and last arrow, $\za$ or $\zb$.  In each case, we need to show that $\rank M_J(V) = \br_J$ if and only if the corresponding rank condition of type (I1) through (I4) on $\zeta(V)$ holds.  Up to a shift, these correspond to which of the four quadrants the southeast corner of a northwest block matrix lies in.

First consider an interval of the form $J=[\za_i,\za_j]$. In this case, $M_J(V)$ is already identical to the northwest submatrix of $\zeta(V)_{y_i \times x_j}$, up to some extra rows and columns of zeros.  So it is clear that the ranks agree.  The same is true for intervals of the form $[\za_i, \zb_n]$.

Now consider an interval of the form $J=[\za_j,\zb_{i-1}]$, where $2 \leq i \leq n$.  To get the rank of $M_J(V)$ from $\zeta(V)$, we must take the northwest submatrix  $\zeta(V)_{x_i \times x_j}$ that includes some identity blocks from the southwest. When computing ranks, these identity blocks clear the columns above them, and add a constant to the rank of the submatrix $M_J(V)$ involved in the definition of quiver rank array.
 
As a concrete example, consider the matrix in Figure \ref{fig:I}. The dashed line outlines the northwest block matrix $\zeta(V)_{{x_{n-1}}\times {x_2}}$, and for any $V$ the rank of this submatrix is
\[ \rank M_{[\alpha_2,\beta_{n-2}]}(V)+\bd(x_n)+\bd(x_{n-1}).\] 

\begin{figure}
\begin{equation}
\begin{pmatrix}
\begin{tikzpicture}[every node/.style={minimum width=1.5em}]
\matrix (m0) [matrix of math nodes,nodes in empty cells]
{
& & & V_{\za_1}& \bid_{\bd(y_0)} &  & \\
& &  V_{\za_2} & V_{\zb_1}& & \bid_{\bd(y_1)} & \\
& \Ddots & \Ddots\\
V_{\za_n} & V_{\zb_{n-1}} &\\
V_{\zb_n} & & && &\phantom{X}&\phantom{X}&\bid_{\bd(y_{n})} \\
\bid_{\bd(x_n)} &  & \\
& \bid_{\bd(x_{n-1})} &  & \\
\phantom{X} &  & \\
\phantom{X} &  & \\
& & &\bid_{\bd(x_{1})} \\
};
\node[scale=3] (zero) at (m0-8-2 -| m0-2-6.south east)  {$0$};
\draw[thick] (m0-6-1.north west) -- (m0-6-1.north west -| m0-5-8.south east);
\draw[thick] (m0-1-5.north west) -- (m0-1-5.north west |- m0-10-4.south east);
\draw[thick,loosely dotted] (m0-7-2) -- (m0-10-4);
\draw[thick,loosely dotted] (m0-2-6) -- (m0-5-8);
\draw[dashed,thick] (m0-1-4.north west) -- (m0-1-4.north west |- m0-7-4.south west);
\draw[dashed,thick] (m0-7-1.south west) -- (m0-1-4.north west |- m0-7-4.south west);
\end{tikzpicture}
\end{pmatrix}
\end{equation}
\caption{Condition (I1)}\label{fig:I}
\end{figure}
The correspondence for other types of intervals can be verified in the same way.
\end{proof}

These three lemmas show that the collection of all northwest block rank conditions is equivalent to the cell, image, and orbit conditions. 

\bibliographystyle{alpha}
\bibliography{bipartite}

\end{document}